%% file: ieeejournal.tex
\documentclass[12pt,onecolumn,journal]{IEEEtran}

\usepackage{amsmath}
\usepackage{amssymb}
\usepackage{amsthm}
\usepackage{bm}
\usepackage{algorithm}
\usepackage{mathrsfs}
\usepackage{algpseudocode} 
\usepackage{enumitem}
\usepackage{verbatim}

\newcommand{\citep}[1]{\cite{#1}}
\newtheorem{definition}{Definition}
\newtheorem{assumption}{Assumption}
\newtheorem{theorem}{Theorem}
\newtheorem{lemma}{Lemma}

\newcommand{\norm}[1]{\|{#1}\|}
\newcommand{\trans}{^{\top}}

\newcommand{\cXs}{\mathcal{X}_{\text{stuck}}}

\usepackage{setspace}
\title{Escaping Saddle Points for Zeroth-order Nonconvex Optimization using Estimated Gradient Descent}

\author{ Qinbo Bai, Mridul Agarwal,  and Vaneet Aggarwal \thanks{The authors are with Purdue University, West Lafayette IN 47907, USA, email:\{bai113,agarw180,vaneet\}@purdue.edu}}

\begin{document}

%

%
\maketitle
\input{intro}

\input{related}

\input{formulation}
\input{gradest}
\input{result}

\input{conclusion}

\appendices

\input{lem5proof}
\input{lem6proof}
\input{lem4proof}
\input{prooflem3}

\bibliographystyle{IEEEtran}


\bibliography{refs}

\end{document}

%% file: intro.tex
\begin{abstract}
Gradient descent and its variants are widely used in machine learning. However, oracle access of gradient may not be available in many applications, limiting the direct use of gradient descent. This paper proposes a method of estimating gradient to perform gradient descent, that converges to a stationary point for general non-convex optimization problems. Beyond the first-order stationary properties, the second-order stationary properties are important in machine learning applications to achieve better performance. We show that the proposed model-free non-convex optimization algorithm returns an $\epsilon$-second-order stationary point with $\widetilde{O}(\frac{d^{2+\frac{\theta}{2}}}{\epsilon^{8+\theta}})$ queries of the function for any arbitrary $\theta>0$.
\end{abstract}

\section{Introduction}

Gradient descent and its variants (e.g., Stochastic Gradient Descent) are widely used in machine learning due to their favorable computational
properties, for example, in optimizing weights of a deep neural network. Given a function $f$: $\mathbb{R}^d\rightarrow\mathbb{R}$, the gradient descent (GD) algorithm updates $\bm{x}_t$ in each iteration as 
\begin{equation}
\bm{x}_{t+1}=\bm{x}_t-\eta\nabla f(\bm{x}_t),
\end{equation}
where $\eta>0$ represent the step size. This algorithm can be shown to achieve $\epsilon$-first-order stationary point for non-convex optimization problem in $O(\frac{1}{\epsilon^2})$ iterations \citep{nesterov1998introductory}. Recently, second order stationary guarantees have been studied by using a perturbed version of gradient descent \citep{jin2017escape}.  However, in many cases, gradient of function may not be accessible and only function value can be queried. This paper studies an algorithm which uses an estimate of the gradient to perform gradient descent, and shows that the algorithm achieves an $\epsilon$-second order stationary point. 

In non-convex settings, convergence to a first-order stationary points is not satisfactory since this point can be a global minima, a local minima, a saddle point, or even a local maxima. Even though finding global minima can be hard, recent results show that, in many problems of interest, all local minima are global minima (e.g., in matrix and tensor completion \citep{jain2013low,liu2016low}, dictionary learning \citep{sun2016complete}, and certain classes of deep neural networks \citep{kawaguchi2016deep}). Saddle points (and local maxima) can correspond to highly suboptimal solutions in many problems \citep{dauphin2014identifying}, where the authors argue that the saddle points are ubiquitous in high-dimensional, non-convex optimization problems, and are thus the main bottleneck in training neural networks. Standard analysis of gradient descent only considers first-order stationary guarantees which do not rule out the saddle points. 

Using stable manifold theorem, authors of \citep{lee2016gradient} prove that gradient descent can indeed escape when the saddle point when initial point is not on the stable manifold. However, they do no provide any complexity analysis for the steps to escape the saddle points. Recently, there has been results based on perturbation of gradient descent to achieve the second-order stationary point \citep{jin2017escape}. However, what happens if the gradient is  not known, which can happen when the function is complex (or available as a black-box) to find the gradient. In this scenario, one approach is to estimate the gradient and perform a gradient descent algorithm. This motivates the question: {\em Can estimated gradient descent escape saddle points and converge to local minima? }

This paper answers this question in positive. We note that this is the first work on the guarantees of gradient-descent based algorithm for zeroth-order non-convex optimization, where only the function can be queried, while the gradient information is not available. Recently, the authors of \citep{balasubramanian2019zeroth} considered the problem of zeroth-order non-convex optimization, while they use cubic regularization based Newton's method. In contrast, we investigate use of regular gradient descent algorithm with the estimation of the gradients.

In this work, without any estimate of the Hessian, {we use the Gaussian smoothening method combined with concentration inequality to give the minimal number of samples we need to estimate the gradient with error at most $\hat{\epsilon}$}. Bounding the error in gradient estimation, we prove that a $\epsilon$-second-order stationary point can be reached with complexity of $O\big(\frac{l(f(\bm{x}_0)-f^*(\bm{x}))}{\epsilon^2}\big)$ iterations following the idea in \citep{jin2017escape}. However, since each iteration queries function multiple times to obtain an estimate of the gradient, the overall complexity is $\widetilde{O}(\frac{d^{2+\frac{\theta}{2}}}{\epsilon^{8+\theta}})$ where $\theta$ is arbitrary positive number. 
The key idea is to use the geometry around saddle points such that the stuck region from which the gradient descent can't escape is a thin band. This means that the small error in the estimation of the gradient in each iteration can lead to escaping this stuck region if the point is not an $\epsilon$-second-order stationary point. Further, the function calls within each iteration can be paralleled decreasing the run-time of the algorithm. 


%% file: related.tex
\section{Related Work}

In recent years, multiple algorithms have been investigated for non-convex optimization problems that converge to $\epsilon$-second-order stationary point. Most of the work has been done for model-based approaches which assume the knowledge of gradient and/or the Hessian of the objective function. Recently, there has also been some work in model-free approaches for non-convex optimization. 

{\bf Model-Based Non-Convex Optimization: } Model-based approaches typically assume the knowledge of derivatives (first or higher order) of the function. We summarize key proposed algorithms on these directions that have been shown to achieve $\epsilon$-second-order stationary point convergence guarantees. 

Based on the knowledge of gradients, the authors of \citep{jin2017escape,jinsept2019} show that the perturbation of gradients in each iteration of the gradient descent can lead to $\epsilon$-second-order stationary point guarantees in $\widetilde{O}(\epsilon^{-2})$ iterations, thus providing no additional loss in the number of iterations required for first order stationary point guarantees. Perturbed versions of stochastic gradient descent have also been studied in \citep{jinsept2019}, where the algorithm finds $\epsilon$-second-order stationary point in ${\tilde O}(\epsilon^{-4})$ iterations if the stochastic gradients are Lipschitz, and ${\tilde O}(d\epsilon^{-4})$ iterations if the stochastic gradients are not Lipschitz.

If the Hessian is also known, one of the approach is to use a successive convex approximation (SCA) method. Perturbation in each iteration of SCA has been shown to achieve $\epsilon$-second-order stationary point in \citep{bedi2019escaping}. Another approach is to add a cubic regularization of Newton method in the iterations \citep{nesterov2006cubic}, where the authors showed that the algorithm can converge to an $\epsilon$-second-order stationary point within $O(\frac{1}{\epsilon^{1.5}})$ gradient and Hessian oracle calls. Recently, stochastic variants of this algorithm have been studied, and have been shown to improve the iterations as compared to stochastic gradient descent \citep{tripuraneni2018stochastic}. Instead of directly querying the Hessian information, recent research shows that one can achieve an $\epsilon$-second-order stationary point using Hessian-vector-product \citep{agarwal2017finding}.

In contrast to these works, we consider a model-free approach, where there is no oracle available to query gradients and/or Hessian. Thus, an estimation of gradient is used to perform gradient descent algorithm. 


{\bf Model-Free Non-Convex Optimization: } A model-free approach to non-convex optimization, also called zeroth-order non-convex optimization assumes that there is an oracle for querying the function. However, anything else about the function (e.g., gradients) is not available. Model-free approaches
for optimization problems estimate the values of gradients and/or Hessians, and are not well understood from a theoretical perspective. Such problems have applications in model-free reinforcement learning \citep{salimans2017evolution} where the objective is not available in closed form and can only be queried. An approach for estimation of gradient has been studied in \citep{conn2009introduction,nesterov2017random}. However, the existing works either find the guarantees for convex optimization or first-order stationary point guarantees. Recently, the authors of \citep{balasubramanian2019zeroth} provided the first model-free algorithm for non-convex optimization with second order guarantees. They use the cubic regularizer of the Newton's method after estimating gradient and Hessian. In contrast, we only estimate the gradient to compute the estimated gradient descent. We also note that the algorithm in  \citep{balasubramanian2019zeroth} requires $O(\frac{d}{\epsilon^{3.5}})+\widetilde{O}(\frac{d^8}{\epsilon^{2.5}})$ function calls while we require $\widetilde{O}(\frac{d^{2+\frac{\theta}{2}}}{\epsilon^{8+\theta}})\approx\widetilde{O}(\frac{d^2}{\epsilon^8})$ function calls to achieve $\epsilon$-second-order stationary point. Thus, our result outperforms that in \citep{balasubramanian2019zeroth} when $d=\Omega(\epsilon^{-(11/12+\delta)})$ for arbitrarily small $\delta>0$.

%% file: formulation.tex
\section{Problem Formulation and Assumptions}
In this section, we will  introduce the notations used in this paper,   describe some  definitions that will be used in this paper, and define the problem formulation formally.

\subsection{Notations}
Bold upper-case letters $\mathbf{A}$ and bold lower-case letters $\bm{x}$ represent the matrices and vectors, respectively. $\bm{x}_i$ denotes the $i^{th}$ element of the vector $\bm{x}$.  $\Vert\cdot\Vert$ is the $l_2$-norm and spectral norm for vectors and matrices, respectively. We use $\lambda_{min}(\cdot)$ to denote the smallest eigenvalue of a matrix.

For a twice-differentiable function $f$: $\mathbb{R}^d\rightarrow\mathbb{R}$, $\nabla f(\cdot)$ and $\nabla^2 f(\cdot)$ are denoted to be the gradient and Hessian of $f$. $f^{*}$ represents the global minimum of the function $f$. $h(n) = O(g(n))$ if and only if there exists a positive real number $M$ and a real number $n_0$ such that ${\displaystyle |h(n)|\leq \;Mg(n){\text{ for all }}n\geq n_{0}.}$ Further, $h(n) = {\widetilde{O}}(g(n))$ if $h(n) = O(g(n) \log^k (g(n)) )$ for any $k>0$. 

$\mathbb{B}^{(d)}_{\bm{x}}(r)$ represents the  ball in $d$ dimension with radius $r$ and center point $\bm{x}$ and we will use $\mathbb{B}_{\bm{x}}(r)$ to simplify the notation when it is clear. $\mathcal{P}_\chi(\cdot)$ is used to denote the projection to the subspace of $\chi$. The norm is assumed to be the Euclidean norm, unless mentioned otherwise. 



\subsection{Definitions }

In this sub-section, we will define a few properties of the function and the stationary point that will be used in the paper. 

\begin{definition} \label{smoothdef}
	A differentiable function $f(\cdot)$ is $l$-smooth if $\forall \bm{x},\bm{y},$
	$$\quad \Vert\nabla f(\bm{x})-\nabla f(\bm{y})\Vert\leq l\Vert\bm{x}-\bm{y}\Vert.$$
\end{definition}

$l$-smooth limits the  speed of increase of the function value. Using the property of $l$-smooth, it is well known that by selecting the stepsize $\eta=\frac{1}{l}$, the gradient descent algorithm will converge within the $O\big(\frac{l(f(\bm{x}_0)-f^*)}{\epsilon^2}\big)$ to the  $\epsilon$-first-order stationary point \citep{jinsept2019}, which is defined as follows. 

\begin{definition} Given a differentiable function $f(\cdot)$, $\bm{x}$ is a first-order stationary point if $\Vert\nabla f(\bm{x})\Vert=0$, and $\bm{x}$ is a $\epsilon$-first-order stationary point if $\Vert\nabla f(\bm{x})\Vert\leq\epsilon$.
\end{definition}

A first order stationary point can be either a local minimum, a local maximum, or a saddle point. In minimization problems, all local maxima and saddle points needs to be avoided. In this paper, we use ``saddle point" to refer to both of them and  are defined as follows:

\begin{definition} Given a differentiable function $f(\cdot)$, $\bm{x}$ is a local minimum if $\exists \epsilon>0$ and $\Vert\bm{y}-\bm{x}\Vert<\epsilon$, we have $f(\bm{x})<f(\bm{y})$. $\bm{x}$ is a ``saddle" point if $\nabla f(\bm{x})=0$ but $\bm{x}$ is not a local minimum. We also define a saddle point $\bm{x}$ to be a strict saddle point if $\lambda_{min}(\nabla^2f(\bm{x}))<0$, which means $\bm{x}$ is non-degenerate.
\end{definition}
	
In this definition, we simply use the word strict saddle point to avoid the degenerate condition where $\lambda_{min}(\nabla^2f(\bm{x}))=0$ and second-order information is not enough to decide the property of $\bm{x}$.

To aviod all strict saddle points in general non-convex problem, we define the $\rho$-Hessian Lipschitz to be as follows.

\begin{definition} \label{defnhessian}
	Given a twice differentiable function $f(\cdot)$, $f$ is $\rho$-Hessian Lipschitz if $\forall\bm{x},\bm{y}$,
	$$\quad \Vert\nabla^2f(\bm{x})-\nabla^2f(\bm{y})\Vert\leq\rho\Vert\bm{x}-\bm{y}\Vert.$$
\end{definition}

The $\rho$-Hessian Lipschitz limits the speed of function increase and also constrains the speed of Hessian matrix changing. Further, we give the definition of $\epsilon$-second-order stationary point, which is the key objective for the proposed algorithm.

\begin{definition}
	Given a $\rho$-Hessian Lipschitz function $f(\cdot)$, $\bm{x}$ is a second-order stationary point if $\Vert\nabla f(\bm{x})\Vert=0$ and $\lambda_{min}(\nabla^2f(\bm{x}))\geq0$. Further,  $\bm{x}$ is a $\epsilon$-second-order stationary point if
	$$\Vert\nabla f(\bm{x})\Vert\leq\epsilon, \quad\quad\lambda_{min}(\nabla^2f(\bm{x}))\geq-\sqrt{\rho\epsilon}.$$
\end{definition}

Finally, we give the definition of the distance between the estimated gradient and true gradient, which is used in the following sections.

\begin{definition}
	Given a differentiable function $f(\cdot)$ and a gradient estimator $\hat{\nabla}$, we say $\hat{\nabla}f(\bm{x})$ is $\hat{\epsilon}$-close to $\nabla f(\bm{x})$ for given point $\bm{x}$ and for some $\hat{\epsilon}>0$ if
	$$\Vert\hat{\nabla}f(\bm{x})-\nabla f(\bm{x})\Vert\leq\hat{\epsilon}.$$
\end{definition}

\subsection{Problem Formulation}

In this paper, we aim to propose an algorithm that is model-free and solves the non-convex optimization problem such that the converged solution is an $\epsilon$-second-order stationary point. We will use  an estimate of the gradient, and perform the gradient descent algorithm. Using this estimated gradient descent, the main aim of the paper is to find the number of iterations required for convergence, as well as the number of function queries needed to converge to an $\epsilon$-second-order stationary point. In order to show the convergence rate, we use the following assumption. 

\begin{assumption}\label{assum}
	Function $f$ is both $l$-smooth and $\rho$-Hessian Lipschitz, and $\Vert\nabla f(\bm{x})\Vert\leq B$ for some finite and positive $B$ for all $\mathbf{x}\in\mathbb{R}^d$.
\end{assumption}

\begin{assumption}\label{symmetric_hessian}
	Hessian matrix, $\nabla^2f(\mathbf{x})$ , of function $f$ is a symmetric matrix for all $\mathbf{x}\in\mathbb{R}^d$. 
\end{assumption}


Based on these assumptions, this paper studies an algorithm of estimating gradient and performing gradient descent based on such estimate, and finds the complexity to return an $\epsilon$-second-order stationary point without any oracle access to the gradient and the Hessian.

%% file: gradest.tex
\section{Proposed Algorithm}

In this section, we will describe the approach used for estimating the gradient, and the estimated gradient descent algorithm that uses this estimate of the gradient. 

\subsection{Estimation of the Gradient}
In zeroth-order oracle, no gradient information is available. To use a gradient descent algorithm, we need to first estimate the gradient information by using the function value. In this subsection, we describe the graident estimation algorithm that is used in this paper. This Algorithm pseudo-code is given in Algorithm \ref{GE}.  

\begin{algorithm}[h]
	\label{GE}
	\caption{Gradient Estimation $GE({d,l,B,c',\hat{\epsilon},\bm{x}})$}
	\begin{algorithmic}[1]
		\State $v\leftarrow\frac{\hat{\epsilon}}{c'l(d+3)^{1.5}}$, $\sigma^2\leftarrow2c'^{2}(d+4)B^2$, $m\leftarrow\frac{32\sigma^2}{\hat{\epsilon}^2}(\log(\frac{1}{\hat{\epsilon}})+\frac{1}{4})$
		\State Generate $\bm{u}_1,...\bm{u}_m$, where $\bm{u}_i\sim\mathcal{N}(0,\mathbf{I}_d)$
		\State $\hat{\nabla}f(\bm{x})=\frac{1}{m}\sum_{i=1}^{m}\frac{f(\bm{x}+v\bm{u}_i)-f(\bm{x})}{v}\bm{u}_i$ 
		\State \Return $\hat{\nabla}f(\bm{x})$
	\end{algorithmic}
\end{algorithm}

The estimation of the gradient uses a Gaussian smoothing approach. Using Gaussian smooth method isn't a new idea. \citep{nesterov2017random} described this method systematically and used this method to give the guarantee for zero-order convex optimization. Despite \citep{balasubramanian2019zeroth} using the similar idea on zeroth-order non-convex optimization, to the best of our knowledge, there is no work that provides the total number of samples required for gradient estimation with error at most $\hat{\epsilon}$. In this paper, we use concentration inequality and conditional probability results to provide such a result which is described formally in Lemma \ref{lemma_GE}.

Recall that $d$ is the dimension of $\bm{x}$,  $l$ is $l$-smooth parameter in Definition \ref{smoothdef}, $B$ is our bound in Assumption \ref{assum} for gradient norm, $c'>1$ is a constant defined in Lemma \ref{lemma_GE}, $\bm{x}$ is the point we make an estimation and $\hat{\epsilon}$ is the intended gap between the estimated gradient and true gradient given in Definition 6. The line 1 in the algorithm gives the parameter used in the following lines. $v$ is the Gaussian smooth parameter. $\sigma^2$ is the bound for the variance of gradient estimator in one sample. $m$ gives the total number of samples we need to get error less than $\hat{\epsilon}$. Line 2 generates $m$ Gaussian random vector with zero mean and variance $\mathbf{I}_d$ which are used to calculate the estimate of the gradient. The estimation algorithm (Line 3) takes an average of estimated gradient using $m$ samples. The next result shows that with an appropriate choice of $m$ and $v$, the estimated gradient is within $\hat{\epsilon}>0$ of the true gradient with probability at least $1-\hat{\epsilon}$. More formally, we have




\begin{lemma} \label{lemma_GE}
	Assume $f(\cdot)$ satisfies Assumption 1. Given an $\hat{\epsilon}>0$, there are fixed constant $c'_{min}$, sample number $m=O(\frac{d}{\hat{\epsilon}^2}\log(\frac{1}{\hat{\epsilon}}))$ and Gaussian smooth parameter $v=\frac{\hat{\epsilon}}{c'l(d+3)^{1.5}}$, such that for $c'>c_{min}$, the estimated gradient,
	$$\hat{\nabla}=\frac{1}{m}\sum_{i=1}^{m}\frac{f(\bm{x}+v\bm{u})-f(\bm{u})}{v}\bm{u},$$
	is $\hat{\epsilon}$-close to $\nabla f(\bm{x})$ with probability at least $1-\hat{\epsilon}$. 
\end{lemma}

\begin{proof}
	For a function $f$ satisfying $l$-smooth, we define Gaussian smooth function $f_v(\bm{x})=\mathbf{E}_{\bm{u}}[f(\bm{x}+v\bm{u})]$, where $\bm{u}$ is a $d$ dimensional standard Gaussian random vector $\bm{u}\sim\mathcal{N}(0,\mathbf{I}_d)$ and $v\in(0,\infty)$ is smooth parameter. Eq. 21 in Section 2 of \citep{nesterov2017random}  shows that
	\begin{equation}
	\nabla f_v(x)=\mathbf{E}_{\bm{u}}[\frac{f(\bm{x}+v\bm{u})-f(\bm{x})}{v}\bm{u}]
	\end{equation}
	We define a gradient estimator $$\hat{\nabla}=\frac{1}{m}\sum_{i=1}^{m}\frac{f(\bm{x}+v\bm{u}_i)-f(\bm{x})}{v}\bm{u}_i,\ \bm{u}_i\sim\mathcal{N}(0,\mathbf{I}_d)$$
	From Lemma 3 and Theorem 4 in \citep{nesterov2017random}, we see that for any function $f$ satisfying $l$-smooth (Notice in the proof of the first inequality in Theorem 4, no convexity is needed), and for any $\bm{x}\in\mathbb{R}^d$, the following hold:
	\begin{equation}\label{fv_bound}
	\Vert\nabla f_v(\bm{x})-\nabla f(\bm{x})\Vert\leq\frac{v}{2}l(d+3)^{\frac{3}{2}}
	\end{equation}
	\begin{equation}\label{var_bound}
		\begin{aligned}
		\frac{1}{v^2}\mathbf{E}_{\bm{u}}&[\{f(\bm{x}+v\bm{u})-f(\bm{x})\}^2\Vert\bm{u}\Vert^2]\\
		&\leq\frac{v^2}{2}l^2(d+6)^{3}+2(d+4)\Vert\nabla f(\bm{x})\Vert^2
		\end{aligned}
	\end{equation}
	To give the distance between $\hat{\nabla}$ and $\nabla f$ is less than $\hat{\epsilon}$, we split the difference to two terms. Here we only consider $0<\hat{\epsilon}<1$.
	\begin{equation*}
	\Vert\hat{\nabla}-\nabla f\Vert\leq\Vert\nabla f_v-\nabla f\Vert+\Vert\hat{\nabla}-\nabla f_v\Vert
	\end{equation*}
	Choosing $v=\frac{\hat{\epsilon}}{c'l(d+3)^\frac{3}{2}}$, where $c'>1$ is a constant will be defined later, we have $\Vert\nabla f_v-\nabla f\Vert\leq\frac{\hat{\epsilon}}{2}$ based on Eq. \eqref{fv_bound}. To bound the second term, noticing that $\mathbf{E}[\hat{\nabla}]=\nabla f_v(\bm{x})$, choose $$\bm{s}_i=\frac{f(\bm{x}+v\bm{u}_i)-f(\bm{x})}{v}\bm{u}_i-\nabla f_v\quad \bm{s}_i'=\bm{s}_i+\nabla f_v$$
	We immediately know $\mathbf{E}[\bm{s}_i]=0$, and the variance of $\bm{s}_i$ can be bounded by
	\begin{equation*}
	\begin{aligned}
	\mathbf{E}[\Vert\bm{s}_i\Vert^2]&=\mathbf{E}[\Vert\bm{s}_i'-\nabla f_v\Vert^2]
	\overset{(a)}= \mathbf{E}[\Vert\bm{s}_i'\Vert^2]-\Vert\nabla f_v\Vert^2\\		
	&\overset{(b)}\leq 2(d+4)B^2+\frac{v^2l^2}{2}(d+6)^3\\
	&\overset{(c)}\leq 2(d+4)B^2+\frac{4\hat{\epsilon}^2}{c'^2}\\
	&\overset{(d)}\leq 2c'^2(d+4)B^2 =: \sigma^2
	\end{aligned}
	\end{equation*}
	where step $(a)$ follows from $\mathbf{E}[\bm{s}_i']=\nabla f_v$. Step $b$ follows from Eq. \eqref{var_bound} and choosing $B>1$. Step $(c)$ holds due to the definition of $v$. Step $(d)$ follows that we omit the term with $\hat{\epsilon}^2$ by multiplying $c'^2>4$ to the first term. Using $l$-smooth, we have
	\begin{equation*}
		\begin{aligned}
		\Vert f(\bm{x}+v\bm{u})-f(\bm{x})\Vert
		\leq vB\Vert \bm{u}\Vert+\frac{lv^2}{2}\Vert\bm{u}\Vert^2
		\end{aligned}
	\end{equation*}
	Thus, the norm of $\bm{s}_i$ can be bounded as:
	\begin{equation}\label{si_bound}
		\begin{aligned}
		\Vert\bm{s}_i\Vert
		&\leq\frac{\Vert f(\bm{x}+v\bm{u})-f(\bm{x})\Vert\Vert\bm{u}\Vert}{v}+B\\
		&\leq B+B\Vert \bm{u}\Vert^2+\frac{lv}{2}\Vert\bm{u}\Vert^3
		\end{aligned}
	\end{equation}
	However, $\bm{u}$ is a Gaussian random vector, there is no bound for it directly. But we can say, given some constant $a\geq 0$, 
	\begin{equation}\label{p_bound}
	P(\Vert\bm{u}\Vert>a)\leq p
	\end{equation} 
	where $p$ is some probability we will calculate in followings.
	Assume $\bm{u}\sim\mathcal{N}(0,\mathbf{I}_d)$, then $\Vert\bm{u}\Vert^2$ follows chi-squared distribution of $d$ degrees of freedom. Consider random variable $e^{t\Vert\bm{u}^2\Vert}$, where $t$ is a constant. For $t>0$,  $e^{t\Vert\bm{u}\Vert^2}$ is strictly increasing with $\Vert\bm{u}\Vert^2$, and using Markov's inequality we obtain,
	\begin{align}
	P(\Vert\bm{u}\Vert^2>a^2)&=P(e^{t\Vert\bm{u}^2\Vert}>e^{ta^2})\leq\frac{\mathbf{E}[e^{t\Vert\bm{u}^2\Vert}]}{e^{ta^2}}\\
	&=\frac{(1-2t)^{-\frac{d}{2}}}{e^{ta^2}} \label{eq:MGF_chi_squared}\\
	&=(1-2t)^{-\frac{d}{2}}e^{-ta^2},\ \forall\ 0<t<\frac{1}{2} \label{norm_bound}
	\end{align}
	Equation \ref{eq:MGF_chi_squared} comes from using the moment generating function of chi-squared distribution of $d$ degrees of freedom.

	Define $f(t)=(1-2t)^{-\frac{d}{2}}e^{-ta^2}$, and choosing $t=\arg\min f(t) = \frac{1}{2}(1-\frac{d}{a^2})$ in Equation \eqref{norm_bound}, we have:
	\begin{equation}\label{P_bound}
		\begin{aligned}
		P(\Vert\bm{u}\Vert^2>a^2)
		&\leq(\frac{d}{a^2})^{-\frac{d}{2}}e^{-\frac{1}{2}(a^2-d)}\\
		&=d^{-\frac{d}{2}}e^{\frac{d}{2}}a^de^{-\frac{a^2}{2}}
		\end{aligned}
	\end{equation}
	For $0<\hat{\epsilon}<1$, we choose $a=c'\cdot\sqrt{\frac{d}{\hat{\epsilon}}}$ so that $t>0$ always holds. Besides, choose $c'>1$ large enough such that $$P(\Vert\bm{u}\Vert^2>a^2)\leq B^{-2}a^{-8}$$
	Now, assuming that $\Vert\bm{u}\Vert\leq a$, combine with Eq. \eqref{si_bound} we have
	\begin{equation*}
	\begin{aligned}
	\Vert\bm{s}_i\Vert
	&\leq B+Ba^2+\frac{lv}{2}a^3=B+\frac{Bc'^2d}{\hat{\epsilon}}+\frac{lvc'^3}{2}\frac{d^{1.5}}{\hat{\epsilon}^{1.5}}\\
	&\leq B+\frac{Bc'^2d}{\hat{\epsilon}}+\frac{c'^2}{2\hat{\epsilon}^{0.5}}\leq\frac{3Bc'^2d}{\hat{\epsilon}}=:\mu
	\end{aligned}
	\end{equation*}
	Combining with Eq. \eqref{p_bound}, we can say given $m$ samples of $\bm{s}_i^{'}$, with probability at least $1-mp$, $\Vert\bm{u}_i\Vert\leq a\ \forall i=1,\cdots,m$. Let $B>1.5$. Based on  Lemma 18 in \citep{2017arXiv170505933K}, we have vector Bernstein Inequality, based on which for $0<\hat{\epsilon}<\frac{\sigma^2}{\mu}=\frac{2(d+4)}{3d}B\hat{\epsilon}$, we have
	\begin{equation*}
	P\big(\Vert\hat{\nabla}-\nabla f_v\Vert\geq\frac{\epsilon}{2}\big)\leq \exp\big(-m\cdot\frac{\epsilon^2}{32\sigma^2}+\frac{1}{4}\big)
	\end{equation*}
	Choosing $m>\frac{32\sigma^2}{\hat{\epsilon}^2}(\log\frac{2}{\hat{\epsilon}}+\frac{1}{4})$, we have $P\big(\Vert\hat{\nabla}-\nabla f_v\Vert\leq\frac{\hat{\epsilon}}{2}\big)\geq 1-\frac{\hat{\epsilon}}{2}$.
	By union bound, the final probability that $\Vert\hat{\nabla}-\nabla\Vert\leq\hat{\epsilon}$ is at least
	\begin{equation*}
		\begin{aligned}
		&1-mp-\frac{\hat{\epsilon}}{2}\\
		&\geq 1-\frac{32{c'}^2(d+4)B^2}{\hat{\epsilon}^2}(\log\frac{1}{\hat{\epsilon}}+\frac{1}{4})\frac{\hat{\epsilon}^4}{{c'}^8d^4B^2}-\frac{\hat{\epsilon}}{2}\\
		&\overset{(a)}\geq 1-(\frac{1}{4}+\frac{1}{4})\hat{\epsilon}-\frac{\hat{\epsilon}}{2}
		\geq 1-\hat{\epsilon}
		\end{aligned}
	\end{equation*}
	By choosing $c'\geq3$, and noting that $\log\frac{1}{\hat{\epsilon}} \leq \frac{1}{\hat{\epsilon}}$. the inequality $(a)$ holds. This completes the proof of the Lemma. 

\end{proof}

Then, based on this result, we run the gradient descent algorithm with estimated gradient in Algorithm \ref{PGD-MF}.

\subsection{Estimated Gradient Descent Algorithm}
This subsection describes the proposed algorithm, which will be analyzed in this paper. 
\begin{algorithm}[h]
	\label{PGD-MF}
	\caption{Estimated Gradient Descent Algorithm $EGD(\bm{x}_0,d,l,B,\chi_1,\theta,\rho,\epsilon,\hat{\epsilon},c,c',\delta,\Delta_f)$}
	\begin{algorithmic}[1]
		\State $\chi \leftarrow \max\{(1+\theta)\log(\frac{2d\ell\Delta_f}{c\epsilon^2\delta}), \chi_1\}$, $\eta\leftarrow\frac{c}{l}$, $g_{thres}\leftarrow\frac{\sqrt{c}}{\chi^2}\cdot\epsilon$, $f_{thres}\leftarrow\frac{c}{\chi^3}\cdot\sqrt{\frac{\epsilon^3}{\rho}}$, $t_{thres}\leftarrow\frac{\chi}{c^2}\cdot\frac{l}{\sqrt{\rho\epsilon}}$, $t_{temp}\leftarrow-t_{thres}-1$, $r\leftarrow \frac{g_{thres}}{l}$
		\For{$t=0,1,...$}
		\State $\hat{\nabla}f(\bm{x}_t)=GE(d,l,B,c',\hat{\epsilon},\bm{x}_t)$
		\If{$\Vert\hat{\nabla}f(\bm{x}_{t})\Vert\leq g_{thres}$, $t-t_{temp}>t_{thres}$}
		\State $\bm{x}_t\leftarrow\bm{x}_t+\bm{\xi}_t$, $\bm{\xi}_t\sim\mathbb{B}^d(r)$
		\State $t_{temp}\leftarrow t$
		\EndIf
		\If{$t-t_{temp}=t_{thres}$ and  $f(\bm{x}_t)-f(\bm{x}_{t-t_{thres}})>-f_{thres}$}
		\State \Return $\bm{x}_{t-t_{thres}}$
		\EndIf
		\State $\bm{x}_{t+1}\leftarrow\bm{x}_t-\eta\hat{\nabla}f(\bm{x}_t)$
		\EndFor
	\end{algorithmic}
\end{algorithm}

The algorithm is described in Algorithm \ref{PGD-MF} and is denoted as EGD. Line 1 gives the input of the algorithm, $\bm{x}_0$ is the initialization point,  $d,l,B,\hat{\epsilon},c'$ are the same defined in algorithm \ref{GE}, $\rho$ is the $\rho$-Hessian parameter as in Definition \ref{defnhessian},  $\theta$ is any constant larger than 0, and $\chi_1$ is the constant so that $\chi_1^3e^{-\chi_1} \le e^{-\chi_1/(1+\theta)}$. We use $\epsilon$ to denote the $\epsilon$-second-order stationary point. $\Delta f$ is a constant so that $\Delta f\geq(f(\bm{x}_0)-f^*)$. $c>0$ is a constant and $\delta>0$ is used such that the  probability of Algorithm \ref{PGD-MF} working correctly is at least $1-\delta$. Due to only zeroth-order information being available, Algorithm \ref{GE} is first used to give an estimate of gradient in each iteration (Line 3). Then the estimated gradient will be used in gradient descent step to replace the unavailable true gradient (Line 11). Besides, the Line (4 - 6) shows that we add a perturbation from a uniformly distributed ball to $\bm{x_t}$ when $\Vert\hat{\nabla}f(\bm{x}_{t})\Vert\leq g_{thres}$ and $t-t_{temp}>t_{thres}$. This means the perturbation will be added when the gradient is small in order to escape the saddle points and it will be added at most once between $t_{thres}$ steps. (Line 8 - 9) checks the terminal condition of the algorithm. If  $f(\bm{x}_t)-f(\bm{x}_{t-t_{thres}})>-f_{thres}$ meaning that the function has not changed enough in the last $t_{thres}$ steps after adding a perturbation, the algorithm immediately returns the point $\bm{x}_{t-t_{thres}}$ as the final result. Our proof in the following section will show that this will indeed lead to an $\epsilon$-second-order stationary point. Thus, this is the condition of the termination for the for-loop. 


%% file: result.tex
\section{Guarantees for the Proposed Algorithm}
In this section, we will show that the proposed algorithm, EGD, returns an $\epsilon$-second-order stationary point. The main result is given as follows. 


\begin{theorem} \label{thm1}
	Assume that $f$ satisfies Assumption \ref{assum}. Then there exists constants $c_{max}$ and $c'_{min}$ such that, for any $\delta>0$, $c\leq c_{max}$, $c'\geq c'_{min}$, $\Delta_f\geq f(\bm{x}_0)-f^{*}$, $\epsilon>0$, $\theta>0$, Let
	$$\hat{\epsilon}\leq\min\{O(\epsilon), \widetilde{O}(\frac{\epsilon^{3+\frac{\theta}{2}}}{d^{\frac{1}{2}(1+\frac{\theta}{2})}})\}$$ $$\chi=\max\{(1+\frac{\theta}{4})\log(\frac{dl\Delta_f}{c\epsilon^2\delta}),\chi_1\}$$
	and $\chi_1$ is a constant such that $\chi_1^3e^{-\chi_1} \le e^{-\chi_1/(1+\frac{\theta}{4})}$ $EGD(\bm{x}_0,d,l,B,\chi_1,\theta,\rho,\epsilon,\hat{\epsilon},c,c',\delta,\Delta_f)$ will output an $\epsilon$-second-order stationary point with probability of $1-\delta$, and terminate in the following number of iterations:
	$$O\big(\frac{l(f(\bm{x}_0)-f^{*})}{\epsilon^2}\log^4\big(\frac{dl\Delta_f}{\epsilon^2\delta}\big)\big)=\widetilde{O}(\frac{1}{\epsilon^2}).$$
	Further,  the number of calls to the function $f(\cdot)$ for the zeroth order algorithm are 
	$$\widetilde{O}(\frac{d^{2+\frac{\theta}{2}}}{\epsilon^{8+\theta}}).$$
\end{theorem}

The rest of the section proves this result. We will first describe the different Lemmas used in the proof, and then use them to give the proof of the theorem.

\subsection{Key Lemmas}
To prove the main result, we first describe two lemmas - Lemma \ref{lemma_geq} and Lemma \ref{lemma_leq}. Lemma  \ref{lemma_geq} indicates that if $\Vert\hat{\nabla}\Vert>g_{thres}$, the function will keep decreasing with the iterations. In other words, we have



\begin{lemma} \label{lemma_geq}
	Assume $f(\cdot)$ satisfies $l$-smooth and $\hat{\nabla}f(\cdot)$ is $\hat{\epsilon}$-close to the $\nabla f(\cdot)$, for any given $\epsilon>0$. Let $\hat\epsilon\leq\frac{\sqrt{c}}{4\chi^2}\cdot
	\epsilon=O(\epsilon)$ and  $c\leq c_{max}$. When $\Vert\hat{\nabla}f(\bm{x}_t)\Vert\geq g_{thres}$, gradient descent with step size $\eta<\frac{1}{l}$ will give
	\begin{equation}
	\label{E2}
	f(\bm{x}_{t+1})\leq f(\bm{x}_t)-\frac{\eta}{4}\Vert\hat{\nabla}f(\bm{x}_t)\Vert
	\end{equation}
\end{lemma}

\begin{proof}
	The result is based on the smoothness property of the function and that the estimated gradient is close to the actual gradient. The steps for the proof can be seen as
\begin{equation}\label{Decrease}
	\begin{aligned}
	&f(\bm{x}_{t+1})\\
	&\overset{(a)}\leq f(\bm{x}_t)+\nabla f(\bm{x}_t)^T(\bm{x}_{t+1}-\bm{x}_t)+\frac{l}{2}\Vert\bm{x}_{t+1}-\bm{x}_t\Vert^2\\
	&\overset{(b)}=f(\bm{x}_t)-\eta\nabla f(\bm{x}_t)^T\hat{\nabla}f(\bm{x}_t)+\frac{l\eta^2}{2}\Vert\hat{\nabla} f(\bm{x}_t)\Vert^2\\
	&=f(\bm{x}_t)-\eta [\nabla f(\bm{x}_t)-\hat{\nabla}f(\bm{x}_t)+\hat{\nabla}f(\bm{x}_t)]   ^T\hat{\nabla}f(\bm{x}_t)\\& \quad +\frac{l\eta^2}{2}\Vert\hat{\nabla} f(\bm{x}_t)\Vert^2\\
	&\overset{(c)}=f(\bm{x}_t)+\eta\Vert\nabla f(\bm{x}_t)-\hat{\nabla}f(\bm{x}_t)\Vert\Vert\hat{\nabla}f(\bm{x}_t)\Vert\\&\quad+\frac{l\eta^2}{2}\Vert\hat{\nabla} f(\bm{x}_t)\Vert^2-\eta\Vert\hat{\nabla} f(\bm{x}_t)\Vert^2\\
	&\overset{(d)}\leq f(\bm{x}_t)+\eta\Vert\nabla f(\bm{x}_t)-\hat{\nabla}f(\bm{x}_t)\Vert\Vert\hat{\nabla}f(\bm{x}_t)\Vert\\&\quad-\frac{\eta}{2}\Vert\hat{\nabla} f(\bm{x}_t) \Vert^2\\
	&\overset{(e)}\leq f(\bm{x}_t)-\frac{\eta}{2}\Vert\hat{\nabla} f(\bm{x}_t) \Vert^2+\eta\hat{\epsilon}\Vert\hat{\nabla}f(\bm{x}_t)\Vert\\
	&\overset{(f)}\leq f(\bm{x}_t)-\frac{\eta}{4}\Vert\hat{\nabla} f(\bm{x}_t) \Vert^2
	\end{aligned}
\end{equation}
	The inequality $(a)$ directly follows from the $l$-smooth property. $(b)$ uses the gradient descent step in Algorithm \ref{PGD-MF}. $(d)$ and $(e)$ holds due to the condition $\eta<\frac{1}{l}$ and $\hat{\nabla}f$ is $\hat{\epsilon}$ close to the $\nabla f$, respectively. Finally, from $\hat{\epsilon}\leq\frac{g_{thres}}{4}\leq\frac{\Vert\hat{\nabla}f(\bm{x}_t)\Vert}{4}$, $(f)$ follows.
\end{proof}

Besides, we note that when $\Vert\hat{\nabla}f(\bm{x})\Vert<g_{thres}$, we have
\begin{equation*}
	\begin{aligned}
	\Vert\nabla f\Vert&=\Vert(\nabla f-\hat{\nabla}f)+\hat{\nabla}f\Vert\\
	&\leq\Vert\nabla f-\hat{\nabla}f\Vert+\Vert\hat{\nabla}f\Vert\\
	&\leq\hat{\epsilon}+\frac{\sqrt{c}}{\chi^2}\epsilon=\frac{5}{4}\frac{\sqrt{c}}{\chi^2}\epsilon\leq\epsilon
	\end{aligned}
\end{equation*} 

By choosing $c<\frac{1}{4}$, the last inequality holds since $\chi>1$. Thus, any $\bm{x}$ satisfying $\Vert\hat{\nabla}f(\bm{x})\Vert<g_{thres}$ is a first order stationary point and  satisfies the first requirement of an $\epsilon$-second-order stationary point.

The next result, Lemma  \ref{lemma_leq}, indicates that if $\Vert\hat{\nabla}f(\widetilde{\bm{x}})\Vert\leq g_{thres}$ and  $\lambda_{min}(\nabla^2f(\widetilde{\bm{x}}))\leq-\sqrt{\rho\epsilon}$, inficating that it is (approximately) first order stationary point with estimated gradient while not (approximately) a second-order stationary point, the proposed algorithm will escape this saddle point by decreasing more than $f_{thres}$ in $t_{thres}$ iterations. 

\begin{lemma} \label{lemma_leq}
	There exist absolute constant $c_{max}$ such that: if $f(\cdot)$ satisfies $l$-smooth and $\rho$-Hessian Lipschitz and any $c\leq c_{max}$, $\hat{\delta}=\frac{dl}{\sqrt{\rho\epsilon}}e^{-\chi}<1$, $$\hat{\epsilon}\leq\min\{O(\epsilon),\widetilde{O}(\frac{\epsilon^{3+\frac{\theta}{2}}}{d^{\frac{1}{2}(1+\frac{\theta}{2})}})\}$$ (we will see $O(\epsilon)$ and $\widetilde{O}(\frac{\epsilon^{3+\frac{\theta}{2}}}{d^{\frac{1}{2}(1+\frac{\theta}{2})}})$ in following lemmas). Let $\eta,r,g_{thres},f_{thres},t_{thres}$ defined as in Algorithm \ref{PGD-MF}. Define $\gamma=\sqrt{\rho\epsilon}$, $\mathcal{T}=\frac{t_{thres}}{c}=\frac{\chi}{\eta\gamma}$ Then if $\widetilde{\bm{x}}$ satisfies:
	$$\Vert\hat{\nabla}f(\widetilde{\bm{x}})\Vert\leq g_{thres}\quad and\quad \lambda_{min}(\nabla^2f(\widetilde{\bm{x}}))\leq-\sqrt{\rho\epsilon}$$
	Let, $\bm{x}_0=\widetilde{\bm{x}}+\bm{\xi}$, where $\bm{\xi}$ comes from the uniform distribution over ball with radius $r=\frac{\sqrt{c}}{\chi^2}\cdot\frac{\epsilon}{l}$. Then with at least probability $1-\hat{\delta}$, we have for  $T=t_{thres}=\frac{\mathcal{T}}{c}$:
	$$f(\bm{x}_{T})-f(\widetilde{\bm{x}})\leq-f_{thres}$$
\end{lemma}

Note that $\delta$ is the probability defined for the algorithm \ref{PGD-MF} and $\hat{\delta}$ is the probability defined for Lemma \ref{lemma_leq}. We first describe the key results to prove this lemma and then give the steps that use these results to prove Lemma \ref{lemma_leq}.

\subsubsection{Key results to prove Lemma \ref{lemma_leq}}

Let $\widetilde{\bm{x}}$ satisfies the conditions in Lemma \ref{lemma_leq}, and without loss of generality let $\bm{e}_1$ be the minimum eigenvector of $\nabla^2f(\widetilde{\bm{x}})$.  Consider two gradient descent sequences $\{\bm{u}_t\}$,$\{\bm{w}_t\}$ with initial points $\bm{u}_0,\bm{w}_0$ satisfying:
$$\Vert\bm{u}_0-\widetilde{\bm{x}}\Vert\leq r,\quad \bm{w}_0=\bm{u}_0+\mu r\bm{e}_1,\mu\in[\frac{\hat{\delta}}{2\sqrt{d}},1].$$ Further, let $\mathcal{P}=\frac{\sqrt{c}}{\chi}\sqrt{\frac{\epsilon}{\rho}}$, $\mathcal{H}=\nabla^2 f(\widetilde{\bm{x}})$, and $\widetilde{f}_{\bm{y}}(\bm{x}):=f(\bm{y})+\nabla f(\bm{y})^T(\bm{x}-\bm{y})+\frac{1}{2}(\bm{x}-\bm{y})^T\mathcal{H}(\bm{x}-\bm{y})$  be a quadratic approximation of $f$ on $\bm{x}$. 

The next result, Lemma \ref{lemma_ut},  shows that if $\Vert\bm{u}_0-\widetilde{\bm{x}}\Vert\leq2r$, we have $\Vert\bm{u}_t-\widetilde{\bm{x}}\Vert\leq100(\mathcal{P}\cdot\hat{c})$ for all $t<T_1$, where $T_1$ is defined in the following result. 


\begin{lemma} \label{lemma_ut}
	Let $f(\cdot),\widetilde{\bm{x}}$ satisfies the conditions in Lemma \ref{lemma_leq}, for any initial point $\bm{u}_0$ with $\Vert\bm{u}_0-\widetilde{\bm{x}}\Vert\leq2r$. Let 	$$T_1=\min\big\{\inf_t\{t\vert\widetilde{f}_{\bm{u}_0}(\bm{u}_t)-f(\bm{u}_0)\leq-4.5f_{thres}\},\hat{c}\mathcal{T}\big\}.$$ 
	Then, there exist absolute constant $c_{max}$ such that for any constant $\hat{c}>3$, $c\leq c_{max}$,  $\hat{\epsilon}\leq\frac{\sqrt{c}}{4\chi^2}\cdot\epsilon=O(\epsilon)$ and $t<T_1$, we have  $\Vert\bm{u}_t-\widetilde{\bm{x}}\Vert\leq100(\mathcal{P}\cdot\hat{c})$.
\end{lemma}

\begin{proof}
	The proof is provided in Appendix \ref{apdlem5}. 
\end{proof}
Let
\begin{equation}\label{eps_cube_bound}
	\hat{\epsilon}\leq\frac{2-\sqrt{2}}{2}\frac{c\sqrt{\epsilon^3\rho}}{\chi^3l}\frac{\hat{\delta}}{2\sqrt{d}}(300\hat{c}+1)=\widetilde{O}(\frac{\epsilon^{3+\frac{\theta}{2}}}{d^{\frac{1}{2}(1+\frac{\theta}{2})}})
\end{equation} $$$$
where $\theta>0$ is a constant we define in theorem \ref{thm1}.
The next result shows that if $\Vert\bm{u}_t-\widetilde{\bm{x}}\Vert\leq100(\mathcal{P}\cdot\hat{c})$, we will have $T_2<\hat{c}\mathcal{T}$, where $T_2$ is as in the statement of the following Lemma. Besides, we will also see how to derive the above $\hat{\epsilon}$ in the proof of this lemma.


\begin{lemma} \label{lemma_wt}
	Let $f(\cdot)$, $\widetilde{\bm{x}}$ satisfy the conditions in Lemma \ref{lemma_leq}. Let $$T_2=\min\{\inf_t\{t\vert\widetilde{f}_{\bm{w}_0}(\bm{w}_t)-f(\bm{w}_0)\leq-4.5f_{thres}\},\hat{c}\mathcal{T}\}.$$
	There are absolute constants $c_{max}$, and $\hat{c}$ such that for any $c\leq c_{max}$, $\hat{\epsilon}$ satisfies Eq. \eqref{eps_cube_bound}, if $\Vert\bm{u}_t-\widetilde{\bm{x}}\Vert\leq100(\mathcal{P}\cdot\hat{c})$ for all $t<T_2$, we have $T_2<\hat{c}\mathcal{T}$ 
\end{lemma}
\begin{proof}
	The proof is provided in Appendix \ref{apdxlem6}. 
\end{proof}

The next result, Lemma \ref{lemma_esc}, combines the two results above to show that  given two gradient descent sequence $\{\bm{u}_t\},\{\bm{w}_t\}$ satisfying the properties given above, at least one of them helps the algorithm decrease the function value greatly.

\begin{lemma} \label{lemma_esc}
	There exist absolute constant $c_{max}$, such that  for any step size $\eta\leq\frac{c_{max}}{l}$, gradient estimation accuracy $\hat{\epsilon}\leq\frac{\sqrt{c^3}}{4\chi}\cdot\epsilon=O(\epsilon)$, and any $T=\frac{\mathcal{T}}{c}$, we have:
	$$\min\{f(\bm{u}_T)-f(\bm{u}_0),f(\bm{w}_T)-f(\bm{w}_0)\}\leq-2.5f_{thres}.$$
\end{lemma}
\begin{proof}
	The proof is given in Appendix \ref{apdlem4}. 
\end{proof}

\subsubsection{Proof of Lemma 3}
\begin{proof}
	Given the result  in  Lemma 6, the proof of Lemma 3 follows on the same lines as Lemma 14 in \citep{jin2017escape}. For  completeness, we provide the detailed steps in Appendix \ref{apdlem3}.
\end{proof}

\input{mainthm}

%% file: mainthm.tex
\subsection{Proof of Theorem \ref{thm1}}

	Choosing $c<\frac{1}{4}$ and starting from $\bm{x}_0$, we consider two cases: 
	\begin{enumerate}
		\item $\norm{\hat{\nabla} f(\bm{x}_0)} > g_{\text{thres}}$: By Lemma \ref{lemma_geq}, we have 
		\begin{equation*}
		f(\bm{x}_{1}) - f(\bm{x}_0) \le  -\frac{\eta}{4} \cdot g_{\text{thres}}^2 = -\frac{c^2}{4\chi^4}\cdot\frac{\epsilon^2}{\ell}
		\end{equation*}	
		\item $\norm{\hat{\nabla} f(\bm{x}_0)} \le g_{\text{thres}}$:
		In this case, Algorithm \ref{PGD-MF} will add a perturbation and check terminal condition after $t_{thres}$ steps. If the condition is not met, with probability at least $1-\hat{\delta}$, we have:
		\begin{equation*}
		f(\bm{x}_{t_{\text{thres}}}) - f(\bm{x}_0) \le -f_{\text{thres}} = -\frac{c}{\chi^3}\cdot\sqrt{\frac{\epsilon^3}{\rho}}
		\end{equation*}
		This means on an average, every step decreases the function value by
		\begin{equation*}
		\frac{f(\bm{x}_{t_{\text{thres}}}) - f(\bm{x}_0)}{t_{\text{thres}}} \le -\frac{c^3}{\chi^4}\cdot\frac{\epsilon^2}{\ell}
		\end{equation*}
	\end{enumerate}
	In Case 1, we can repeat this argument for $t = 1$. In Case 2, we can repeat this argument for $t=t_{thres}+1$.
	Since we choose $c_{max}<\frac{1}{4}$,  the gradient descent will decrease function value in each iteration by at least $\frac{c^3}{\chi^4}\cdot\frac{\epsilon^2}{\ell}$. However, the function value can't be decreased by more than $f(\bm{x}_0) - f^*$, where $f^*$ is the function value of global minima. This means algorithm \ref{PGD-MF} must terminate within the following number of iterations:
	\begin{equation*}
	\begin{aligned}
	\frac{f(\bm{x}_0) - f^*}{\frac{c^3}{\chi^4}\cdot\frac{\epsilon^2}{\ell}}
	&= \frac{\chi^4}{c^3}\cdot \frac{\ell(f(\bm{x}_0) - f^*)}{\epsilon^2} \\
	&= O\left(\frac{\ell(f(\bm{x}_0) - f^*)}{\epsilon^2}\log^{4}\left(\frac{d\ell\Delta_f}{\epsilon^2\delta}\right) \right)
	\end{aligned}
	\end{equation*}
	
	Recall that our choice for $\hat{\epsilon}\leq\widetilde{O}(\frac{\epsilon^{3+\frac{\theta}{2}}}{d^{\frac{1}{2}(1+\frac{\theta}{2})}})$. The number of function evaluations of Algorithm \ref{PGD-MF} as a function of parameters $d$ and $\epsilon$ is given as

	$$O(\frac{1}{\epsilon^2}\log^4\big(\frac{d}{\epsilon^2}\big)\cdot\frac{d}{\hat{\epsilon}^2}\log\frac{1}{\hat\epsilon})=\widetilde{O}(\frac{d}{\epsilon^2\hat{\epsilon}^2})=\widetilde{O}(\frac{d^{2+\frac{\theta}{2}}}{\epsilon^{8+\theta}}).$$

	
	
	Finally, we give the probability of obtaining an $\epsilon$-second order stationary point when the gradient descent algorithm stops. 
	According to Lemma \ref{lemma_geq}, the function value always decreases in case 1. By Lemma \ref{lemma_leq}, we know the function value decreases with probability at least $1-\frac{d\ell}{\sqrt{\rho\epsilon}}e^{-\chi}$ each time the algorithm meets case 2. Besides, we know the number of times we check the terminal condition during the process of gradient descent is at most:
	\begin{equation*}
	\frac{1}{t_{\text{thres}}} \cdot \frac{\chi^4}{c^3}\cdot \frac{\ell(f(\bm{x}_0) - f^*)}{\epsilon^2}
	=\frac{\chi^3}{c}\frac{\sqrt{\rho\epsilon}(f(\bm{x}_0) - f^*)}{\epsilon^2}
	\end{equation*}
	Besides, by Lemma \ref{lemma_GE}, we know the probability of $\Vert\hat{\nabla}-\nabla\Vert\leq\hat{\epsilon}$ is at least $1-\hat{\epsilon}$ each time we make a estimation. And the number of estimation is given by the number of iteration $\frac{\chi^4}{c^3}\cdot\frac{l\Delta_f}{\epsilon^2}$.
	Thus, by union bound, we bound this two probability together to give the final probability of the Algorithm \ref{PGD-MF} giving an $\epsilon$-second order stationary point is at least:
	\begin{equation*}
	\begin{aligned}
	1- &\frac{d\ell}{\sqrt{\rho\epsilon}}e^{-\chi} \cdot \frac{\chi^3}{c}\frac{\sqrt{\rho\epsilon}(f(\bm{x}_0) - f^*)}{\epsilon^2} - \hat{\epsilon}\cdot\frac{\chi^4l\Delta_f}{c^3\epsilon^2}\\
	&= 1 - \frac{\chi^3e^{-\chi}}{c}\cdot  \frac{d\ell(f(\bm{x}_0) - f^*)}{\epsilon^2} -\hat{\epsilon}\cdot\frac{\chi^4l\Delta_f}{c^3\epsilon^2}
	\end{aligned}
	\end{equation*}
	
	Recall our choice of $\chi = \max\{(1+\frac{\theta}{4})\log(\frac{2d\ell\Delta_f}{c\epsilon^2\delta}), \chi_1\}$, where $\theta>0$, we have $\chi_1^3e^{-\chi_1} \le e^{-\chi_1/(1+\frac{\theta}{4})}$, and $\hat{\epsilon}\leq \widetilde{O}(\epsilon^3)$ this gives the probability of the Alforithm not resulting in an $\epsilon$-second order stationary point is at most
	\begin{equation*}
	\begin{aligned}
	\frac{\chi^3e^{-\chi}}{c}&\cdot  \frac{d\ell(f(\bm{x}_0) - f^*)}{\epsilon^2} + \hat{\epsilon}\cdot\frac{\chi^4l\Delta_f}{c^3\epsilon^2}\\
	&\le e^{-\chi/(1+\frac{\theta}{4})}  \frac{d\ell(f(\bm{x}_0) - f^*)}{c\epsilon^2} +\frac{\delta}{2}\le \delta
	\end{aligned}
	\end{equation*}
	which finishes the proof of the Theorem.
	

%% file: conclusion.tex
\section{Conclusion}

This paper proposea a Perturbed Estimated Gradient Descent Algorithm with only access to the zeroth-order information of objective function. With only estimated gradient information, we prove the second-order stationary point convergence of the algorithm and provide the convergence rate. This is the first result, to the best of our knowledge, that provides the convergence rate results of gradient descent based method for achieving $\epsilon$-second order stationary point with zeroth-order information. 

In the proposed algorithm, we use a perturbation of the estimated gradient descent, where the perturbation was needed to escape the first order stationary point that is not a second order stationary point. However, it may be possible that the estimation error controlled through Gaussian smoothening alone helps escape saddle points. Whether the additional perturbation in the algorithm can be removed is a topic of future work. 

%% file: lem5proof.tex
\section{Proof of Lemma \ref{lemma_ut}}\label{apdlem5}
	\begin{proof}
	Without loss of generality, we set $\bm{u}_0=0$ to be the origin, by the update function, we have:
	\begin{equation}
	\label{E3}
	\begin{aligned}
	\bm{u}_{t+1}&=\bm{u}_t-\eta\hat{\nabla}f(\bm{u}_t)\\
	&=\bm{u}_t-\eta\nabla f(\bm{u}_t)-\eta[\hat{\nabla}f(\bm{u}_t)-\nabla f(\bm{u}_t)]\\		
	&=\bm{u}_t-\eta\nabla f(0)-\eta\left[\int_{0}^{1}\nabla^2f(\theta\bm{u}_t)d\theta\right]\bm{u}_t-\eta[\hat{\nabla}f(\bm{u}_t)-\nabla f(\bm{u}_t)]\\
	&=\bm{u}_t-\eta\nabla f(0)-\eta(\mathcal{H}+\Delta_t)\bm{u}_t-\eta[\hat{\nabla}f(\bm{u}_t)-\nabla f(\bm{u}_t)]\\
	&=(\mathbf{I}-\eta\mathcal{H}-\eta\Delta_t)\bm{u}_t-\eta\nabla f(0)-\eta[\hat{\nabla}f(\bm{u}_t)-\nabla f(\bm{u}_t)]\\
	\end{aligned}
	\end{equation}
	where $\Delta_t=\int_{0}^{1}\nabla^2f(\theta\bm{u}_t)d\theta-\mathcal{H}$ can be bounded as:
	\begin{equation}\label{Delta_bound}
		\begin{aligned}
		\Vert\Delta_t\Vert&=\Vert\int_{0}^{1}\nabla^2f(\theta\bm{u}_t)d\theta-\mathcal{H}\Vert\\
		&\leq\int_{0}^{1}\Vert\nabla^2f(\theta\bm{u}_t)-\nabla^2f(\widetilde{\bm{x}})\Vert d\theta\\
		&\leq\int_{0}^{1}\rho\Vert\theta\bm{u}_t-\widetilde{\bm{x}}\Vert d\theta\\
		&\leq\rho\int_{0}^{1}\theta\Vert\bm{u}_t\Vert+\Vert\widetilde{\bm{x}}\Vert d\theta \leq\rho(\Vert\bm{u}_t\Vert+\Vert\widetilde{\bm{x}}\Vert)
		\end{aligned}
	\end{equation}
	Besides, based on $l$-smooth, we have $\Vert\nabla f(0)\Vert\leq\Vert\nabla f(\widetilde{\bm{x}})\Vert+l\Vert\widetilde{\bm{x}}\Vert\leq g_{thres}+2lr=3g_{thres}$.\\
	Now let $\mathcal{S}$ to be the spaced spanned by the eigenvectors of $\mathcal{H}$ whose eigenvalue is less than $-\frac{\gamma}{\hat{c}\chi}$. Let $\mathcal{S}^c$ to be the space spanned by the other eigenvectors. Let $\bm{\alpha}_t$ and $\bm{\beta}_t$ denote the projections of $\bm{u_t}$ onto $\mathcal{S}$ and $\mathcal{S}^c$. According to Eq. \ref{E3}, we have
	\begin{equation}
	\label{beta_t}
	\bm{\beta}_{t+1}=(\mathbf{I}-\eta\mathcal{H})\bm{\beta}_t-\eta\mathcal{P_S}^c\Delta_t\bm{u}_t-\eta\mathcal{P_S}^c\nabla f(0)-\eta\mathcal{P_S}^c[\hat{\nabla}f(\bm{u}_t)-\nabla f(\bm{u}_t)]
	\end{equation}
	By the definition of $T_1$ in lemma \ref{lemma_ut}, for all $t<T_1$
	\begin{equation}
		\label{bound_ut}
		-4.5f_{thres}<\widetilde{f}_0(\bm{u}_t)-f(0)=\nabla f(0)^T\bm{u}_t+\frac{1}{2}\bm{u}_t^T\mathcal{H}\bm{u}_t\leq\nabla f(0)^T\bm{u}_t-\frac{\gamma}{2}\frac{\Vert\bm{\alpha}_t\Vert^2}{\hat{c}\chi}+\bm{\beta}_t\mathcal{H}\bm{\beta}_t
	\end{equation}
	To see the last inequality, we define the orthogonal eigenvectors in $\mathcal{S}$ and $\mathcal{S}^c$ are $\bm{\alpha}^1,\bm{\alpha^2},...,\bm{\alpha^m}$ and $\bm{\beta}^1,\bm{\beta}^2,...,\bm{\beta}^n$, where $d=m+n$. Thus, $\bm{u}_t=\bm{\alpha_t}+\bm{\beta_t}=a_1\bm{\alpha}^1+a_2\bm{\alpha}^2+...+a_m\bm{\alpha}^m+b_1\bm{\beta}^1+b_2\bm{\beta^2}+...+b_n\bm{\beta}^n$, where $a_1,...a_m,b_1,...b_n$ are the linear combination parameter, and the eigenvalues for eigenvectors $\bm{\alpha}^1,...\bm{\alpha}^m\leq-\frac{\gamma}{\hat{c}\chi}$ by the definition of the space $\mathcal{S}$. Thus, we have
	\begin{equation*}
		\begin{aligned}
		\bm{u}_t^T\mathcal{H}\bm{u}_t&=\bm{u}_t^T\mathcal{H}(a_1\bm{\alpha}^1+a_2\bm{\alpha}^2+...+a_m\bm{\alpha}^m+b_1\bm{\beta}^1+b_2\bm{\beta^2}+...+b_n\bm{\beta}^n)\\
		&\leq-\frac{\gamma}{\hat{c}\chi}\bm{u}_t^T(a_1\bm{\alpha}^1+a_2\bm{\alpha}^2+...+a_m\bm{\alpha}^m)+\bm{u}_t^T\mathcal{H}\bm{\beta_t}\\
		&\leq-\frac{\gamma}{\hat{c}\chi}\Vert\bm{\alpha}_t\Vert^2+\bm{\beta}_t^T\mathcal{H}\bm{\beta}_t
		\end{aligned}
	\end{equation*}
	where the last step use the orthogonality of $\bm{\alpha_t}$ and $\bm{\beta}_t$\\
	According to $\Vert\bm{u}_t^2\Vert=\Vert\bm{\alpha}_t^2\Vert+\Vert\bm{\beta}_t^2\Vert$, noticing that $\Vert\nabla f(0)\Vert\leq 3g_{thres}$, combing with Eq. \eqref{bound_ut}, we have,
	\begin{equation*}
	\begin{aligned}
	\Vert\bm{u}_t\Vert^2 &\leq\frac{2\hat{c}\chi}{\gamma}\left(4.5f_{thres}+\nabla f(0)^T\bm{u}_t+\bm{\beta}_t\mathcal{H}\bm{\beta}_t\right)+\Vert\bm{\beta}_t^2\Vert\\
	&\leq17\cdot\max\big\{\frac{g_{thres}\hat{c}\chi}{\gamma}\Vert\bm{u}_t\Vert,\frac{f_{thres}\hat{c}\chi}{\gamma},\frac{\bm{\beta}_t\mathcal{H}\bm{\beta}_t\hat{c}\chi}{\gamma},\Vert\bm{\beta_t}\Vert^2\big\} 
	\end{aligned}
	\end{equation*}
	Which means,
	\begin{equation}\label{ut_bound}
		\begin{aligned}
		\Vert\bm{u}_t\Vert&\leq17\cdot\max\big\{\frac{g_{thres}\hat{c}\chi}{\gamma},\sqrt{\frac{f_{thres}\hat{c}\chi}{\gamma}},\sqrt{\frac{\bm{\beta}_t\mathcal{H}\bm{\beta}_t\hat{c}\chi}{\gamma}},\Vert\bm{\beta_t}\Vert\big\}\\
		&=17\cdot\max\big\{\hat{c}\cdot\mathcal{P},\hat{c}\cdot\mathcal{P},\sqrt{\frac{\bm{\beta}_t\mathcal{H}\bm{\beta}_t\hat{c}\chi}{\gamma}},\Vert\bm{\beta_t}\Vert\big\}
		\end{aligned}
	\end{equation}
	The last equality is due to the definition of $g_{thres}$ and $f_{thres}$. Now, we use induction to prove for all $t<T_1$, we have $\Vert\bm{u}_t\Vert\leq100(\mathcal{P}\cdot\hat{c})$.
	According to the Eq. \eqref{ut_bound}, we only need to use induction on the last two terms.
	When $t=0$, it is obvious due to $\bm{u_0}=0$, suppose the induction holds when $\tau=t<T_1$, we will show that it still holds for $\tau=t+1<T_1$, Let $$\bm{\delta}_t=\mathcal{P_S}^c\left[-\Delta_t\bm{u}_t-\nabla f(0)-(\hat{\nabla}f(\bm{u}_t)-\nabla f(\bm{u}_t)) \right]$$ 
	By Eq. \eqref{beta_t}, define $\kappa=\frac{l}{\gamma}>1$, we have
	\begin{equation}
	\label{delta_dyn}
	\bm{\beta_{t+1}}=(\mathbf{I}-\eta\mathcal{H})\bm{\beta_t}+\eta\bm{\delta_t}
	\end{equation}
	and we can bound $\bm{\delta}_t$ as
	\begin{equation} \label{delta_bound}
		\begin{aligned}
		\Vert\bm{\delta}_t\Vert
		&\leq\Vert\Delta_t\Vert\Vert\bm{u}_t\Vert+\Vert\nabla f(0)\Vert+\Vert\hat{\nabla}f(\bm{u}_t)-\nabla f(\bm{u}_t)\Vert\\
		&\overset{(a)}\leq \rho(\Vert\bm{u}_t\Vert+\Vert\widetilde{\bm{x}}\Vert)\Vert\bm{u}_t\Vert+\Vert\nabla f(0)\Vert+\hat{\epsilon}\\
		&\overset{(b)}\leq\rho\cdot100\hat{c}(100\hat{c}\mathcal{P}+2r)\mathcal{P}+\frac{5}{4}g_{thres}\\
		&=100\hat{c}(100\hat{c}+\frac{2}{\chi k})\rho\mathcal{P}^2+\frac{5}{4}g_{thres}\\
		&\overset{(c)}\leq[100\hat{c}(100\hat{c}+2)\sqrt{c}+\frac{5}{4}]g_{thres}
		\overset{(d)}\leq 1.5g_{thres}
		\end{aligned}
	\end{equation}
	where $(a)$ uses Eq. \eqref{Delta_bound}, $(b)$ uses the induction assumption when $\tau=t$, $\rho\mathcal{P}=\rho(\frac{c}{\chi^2}\cdot\frac{\epsilon}{\rho})=\sqrt{c}(\frac{\sqrt{c}}{\chi^2}\cdot\epsilon)=\sqrt{c}g_{thres}$ gives the step $(c)$.
	By choosing $c_{max}\leq\frac{1}{4}\frac{1}{100\hat{c}(100\hat{c}+2)}$ and step size $c\leq c_{max}$, the last inequality $(d)$ holds.
	\paragraph{Bounding $\norm{\bm{\beta}_{t+1}}$:}
	Combining Eq.\eqref{delta_dyn}, Eq.\eqref{delta_bound} and using the definition of $\mathcal{S}^c$, we have:
	\begin{equation*}
	\norm{\bm{\beta}_{t+1}} \le (1+ \frac{\eta \gamma}{\hat{c}\chi}) \norm{\bm{\beta}_t} + 1.5\eta g_{thres}
	\end{equation*}
	Since $\norm{\bm{\beta}_0} = 0$ and $t+1 \le T_1$, by applying above relation recursively, we have:
	\begin{equation}\label{beta_bound}
	\norm{\bm{\beta}_{t+1}} 
	\le \sum_{\tau = 0}^{t}1.5(1+ \frac{\eta \gamma}{\hat{c}\chi})^\tau\eta g_{thres} 
	\overset{(a)}\le 1.5\cdot 3\cdot T_1\eta g_{thres}
	\overset{(b)}\le 5(\mathcal{P} \cdot \hat{c})
	\end{equation}
	Step $(a)$ holds because $T_1\le \hat{c} \mathcal{T}=\frac{\eta\gamma}{c\chi}$ by definition, so that $(1+ \frac{\eta \gamma}{\hat{c}\chi})^{T_1}\le 3$. And step $(b)$ holds because $T_1\leq\hat{c}\mathcal{T}\eta g_{thres}=\hat{c}\frac{\chi}{\eta\gamma}\eta\frac{\sqrt{c}}{\chi^2}\epsilon=\hat{c}\frac{\sqrt{c}}{\chi}\sqrt{\frac{\epsilon}{\rho}}=\hat{c}\mathcal{P}$
	
	\paragraph{Bounding $\bm{\beta}_{t+1}\trans\mathcal{H}\bm{\beta}_{t+1}$:} Using Eq.\eqref{delta_dyn}, we can also write the update equation as:
	\begin{equation*}
	\begin{aligned}
	\bm{\beta_t} &= (\mathbf{I}-\eta\mathcal{H})\bm{\beta_{t-1}}+\eta\bm{\delta_{t-1}}\\
	&=(\mathbf{I}-\eta\mathcal{H})[(\mathbf{I}-\eta\mathcal{H})\bm{\beta_{t-2}}+\eta\bm{\delta_{t-2}}]+\eta\bm{\delta_{t-2}}\\
	&=(\mathbf{I}-\eta\mathcal{H})^2\bm{\beta_{t-2}}+(\mathbf{I}-\eta\mathcal{H})\eta\bm{\delta_{t-2}}+\eta\bm{\delta_{t-1}}\\
	&=...\\
	&=\sum_{\tau=0}^{t-1}(\mathbf{I}-\eta\mathcal{H})^\tau\eta\bm{\delta_{t-\tau-1}}
	\end{aligned}
	\end{equation*}
	Combining with Eq.\eqref{delta_bound}, this gives
	\begin{equation}\label{beta2_bound}
		\begin{aligned}
		\bm{\beta}_{t+1}\trans \mathcal{H} \bm{\beta}_{t+1} =& \eta^2\sum_{\tau_1 = 0}^t \sum_{\tau_2 = 0}^t 
		\bm{\delta_{t-\tau_1}}\trans (\mathbf{I} - \eta \mathcal{H})^{\tau_1}\mathcal{H}(\mathbf{I} - \eta \mathcal{H})^{\tau_2}\bm{\delta_{t-\tau_2}} \\
		\le &\eta^2\sum_{\tau_1 = 0}^t \sum_{\tau_2 = 0}^t \norm{\bm{\delta_{t-\tau_1}}}
		\norm{(\mathbf{I} - \eta \mathcal{H})^{\tau_1}\mathcal{H}(\mathbf{I} - \eta \mathcal{H})^{\tau_2}}\norm{\bm{\delta_{t-\tau_2}}} \\
		\le& 4\eta^2 g_{thres}^2\sum_{\tau_1 = 0}^t \sum_{\tau_2 = 0}^t  \norm{(\mathbf{I} - \eta \mathcal{H})^{\tau_1}\mathcal{H}(\mathbf{I} - \eta \mathcal{H})^{\tau_2}}
		\end{aligned}
	\end{equation}
	Let the eigenvalues of $\mathcal{H}$ to be $\{\lambda_i\}$, then for any $\tau_1, \tau_2 \ge 0$, we know the eigenvalues of 
	$(\mathbf{I} - \eta \mathcal{H})^{\tau_1}\mathcal{H}(\mathbf{I} - \eta \mathcal{H})^{\tau_2}$ are $\{\lambda_i(1-\eta \lambda_i)^{\tau_1 + \tau_2}\}$.
	Let $g_t(\lambda):=\lambda (1-\eta \lambda)^t$, and setting its derivative to zero, we obtain:
	\begin{equation*}
	g_t(\lambda)' = (1-\eta\lambda)^t -t\eta\lambda(1-\eta\lambda)^{t-1} = 0
	\end{equation*}
	Because $l$ is the largest eigenvalue of Hessian, we must have $\lambda\leq l=\frac{c}{\eta}\leq\frac{1}{\eta}$. Thus,
	we see that $\lambda_t^\star = \frac{1}{(1+t)\eta}$ is the unique maximizer, and $g_t(\lambda)$ is monotonically increasing in $(-\infty, \lambda_t^\star]$. This gives:
	\begin{equation*}
	\norm{(\mathbf{I} - \eta \mathcal{H})^{\tau_1}\mathcal{H}(\mathbf{I} - \eta \mathcal{H})^{\tau_2}}
	= \max_i \lambda_i(1-\eta \lambda_i)^{\tau_1 + \tau_2}
	\le \hat{\lambda}(1-\eta\hat{\lambda})^{\tau_1 + \tau_2} \le \frac{1}{(1+\tau_1+\tau_2)\eta}
	\end{equation*}
	where $\hat{\lambda} = \min\{l, \lambda_{\tau_1 + \tau_2}^\star\}$. Using this equation in Eq. \eqref{beta2_bound}, we have
	\begin{equation}\label{beta2_bound2}
	\begin{aligned}
	\bm{\beta}_{t+1}\trans \mathcal{H} \bm{\beta}_{t+1} 
	&\le 4\eta^2 g_{thres}^2\sum_{\tau_1 = 0}^t \sum_{\tau_2 = 0}^t  \norm{(\mathbf{I} - \eta \mathcal{H})^{\tau_1}\mathcal{H}(\mathbf{I} - \eta \mathcal{H})^{\tau_2}}  \\
	&\le 4\eta g_{thres}^2\sum_{\tau_1 = 0}^t \sum_{\tau_2 = 0}^t \frac{1}{1+\tau_1+\tau_2}\\
	&\overset{(a)}\le 8\eta T_1 g_{thres}^2 
	\overset{(b)}\le 8\eta\hat{c}\mathcal{T}g_{thres}^2
	\overset{(c)}= 8 \mathcal{P}^2 \gamma \hat{c} \cdot \chi^{-1} 
	\end{aligned}
	\end{equation}
	where step $(a)$ holds by rearranging the summation as follows:
	\begin{equation*}
	\sum_{\tau_1 = 0}^t \sum_{\tau_2 = 0}^t \frac{1}{1+\tau_1+\tau_2}
	= \sum_{\tau = 0}^{2t} \min\{1+\tau, 2t+1-\tau\} \cdot \frac{1}{1+\tau} \le 2t+1 < 2T_1
	\end{equation*}
	and step $(b)$ use the definition $T_1\leq\hat{c}\mathcal{T}$ and $\eta\mathcal{T}g_{thres}^2=\eta\frac{\chi}{\eta\gamma}\frac{c}{\chi^4}\epsilon^2=\frac{c\epsilon^2}{\gamma\chi^3}=(\frac{c}{\chi^2}\frac{\epsilon}{\rho})\gamma\chi^{-1}=\mathcal{P}^2\gamma\chi^{-1}$ give the result of step $(c)$ 
	~
	
	Finally, substituting Eq. \eqref{beta_bound} and Eq. \eqref{beta2_bound2} into Eq.\eqref{ut_bound}, we have
	\begin{align*}
	\norm{\bm{u}_{t+1}}  \le& 17\cdot\max\big\{\hat{c}\cdot\mathcal{P},\hat{c}\cdot\mathcal{P},\sqrt{\frac{\bm{\beta}_t\mathcal{H}\bm{\beta}_t\hat{c}\chi}{\gamma}},\Vert\bm{\beta_t}\Vert\big\}\\
	\le & 100 (\mathcal{P} \cdot \hat{c})
	\end{align*}
	This finishes the induction as well as the proof of the lemma. \hfill 
\end{proof}

%% file: lem6proof.tex
\section{Proof of Lemma \ref{lemma_wt}}\label{apdxlem6}

\begin{proof}	
	
	In this lemma, we will show that if sequence $\bm{u}_t$ is inside a small ball, then the sequence $\bm{w}_t$ can escape the stuck region. To see this, we focus on the difference of these two sequence in direction $\bm{e}_1$. We will prove that the different in $\bm{e}_1$ direction is increase as power series with base larger than 1. In this sense, it won't take long to get sequence $\bm{w}_t$ escaping the stuck region.\\
	W.L.O.G, set $\bm{u}_0 = 0$ to be the origin. Define $\bm{v}_t = \bm{w}_t - \bm{u}_t$, by assumptions in Lemma \ref{lemma_leq}, we have $\bm{v}_0 = \mu r\bm{e}_1, ~\mu \in [\hat{\delta}/(2\sqrt{d}), 1]$. Now, consider the update equation for $\bm{w}_t$:
	\begin{align*}
	\bm{u_{t+1}+\bm{v}_{t+1}}=\bm{w}_{t+1}&=\bm{w}_t-\eta\hat{\nabla}f(\bm{w}_t)\\
	&=\bm{u}_t+\bm{v}_t-\eta\nabla f(\bm{u}_t+\bm{v}_t)+\eta(\hat{\nabla}f(\bm{w}_t)-\nabla f(\bm{w}_t))\\
	&=\bm{u}_t+\bm{v}_t-\eta\nabla f(\bm{u}_t)-\eta\big[\int_{0}^{1}\nabla^2f(\bm{u}_t+\theta\bm{v}_t)d\theta\big]\bm{v}_t+\eta(\hat{\nabla}f(\bm{w}_t)-\nabla f(\bm{w}_t))\\
	&=\bm{u}_t+\bm{v}_t-\eta\nabla f(\bm{u}_t)-\eta(\mathcal{H}+\Delta_t^{'})\bm{v}_t+\eta(\hat{\nabla}f(\bm{w}_t)-\nabla f(\bm{w}_t))\\
	&=\bm{u}_t-\eta\nabla f(\bm{u}_t)+(\mathbf{I}-\eta\mathcal{H}-\eta\Delta_t^{'})\bm{v}_t+\eta(\hat{\nabla}f(\bm{w}_t)-\nabla f(\bm{w}_t))\\
	&=\bm{u}_{t+1}+(\mathbf{I}-\eta\mathcal{H}-\eta\Delta_t^{'})\bm{v}_t+\eta(\hat{\nabla}f(\bm{w}_t)-\nabla f(\bm{w}_t))+\eta(\hat{\nabla}f(\bm{u}_t)-\nabla f(\bm{u}_t))
	\end{align*}
	where $\Delta'_t := \int_{0}^1 \nabla^2 f(\bm{u}_t + \theta\bm{v}_t)d\theta - \mathcal{H}$. By Hessian Lipschitz, similar to Lemma \ref{lemma_ut}, we have $\norm{\Delta'_t} \le \rho(\norm{\bm{u}_t} + \norm{\bm{v}_t}+ \norm{\widetilde{\bm{x}}})$. 
	Thus, $\bm{v}_t$ satisfies
	\begin{equation}\label{v_t_dyn}
	\bm{v}_{t+1} = (\mathbf{I} - \eta \mathcal{H} - \eta \Delta'_t) \bm{v}_t+\eta(\hat{\nabla}f(\bm{w}_t)-\nabla f(\bm{w}_t))+\eta(\hat{\nabla}f(\bm{u}_t)-\nabla f(\bm{u}_t))
	\end{equation}
	
	Since $\Vert\bm{w}_0-\widetilde{\bm{x}}\Vert=\Vert\bm{u}_0-\widetilde{\bm{x}}+\bm{v}_0\Vert\leq\Vert\bm{u}_0-\widetilde{\bm{x}}\Vert+\Vert\bm{v}_0\Vert\leq2r$ by definition of $\bm{u}_0$, directly applying Lemma \ref{lemma_ut}, we obtain $\bm{w}_t\le 100(\mathcal{P}\cdot\hat{c})$ for all $t \le T_2$. By condition of Lemma \ref{lemma_wt}, we obtain $\norm{\bm{u}_t} \le 100(\mathcal{P}\cdot\hat{c})$ for all $t<T_2$.
	This gives:
	\begin{equation} \label{vt_bound}
	\norm{\bm{v}_t} \le \norm{\bm{u}_t} + \norm{\bm{w}_t} \le 200(\mathcal{P}\cdot\hat{c}) \text{~for all~} t<T_2
	\end{equation}
	Thus, for $t<T_2$, we have:
	\begin{equation*}
	\norm{\Delta'_t} \le \rho( \norm{\bm{u}_t} + \norm{\bm{v}_t}+ \norm{\widetilde{\bm{x}}})
	\le \rho(300\mathcal{P}\cdot\hat{c}+r)
	=\rho\mathcal{P}(300\hat{c}+\frac{1}{\chi\kappa})
	\le \rho\mathcal{P}(300\hat{c}+1)
	\end{equation*}

	Denote $\psi_t\geq0$ as the norm of $\bm{v}_t$ projected onto $\bm{e}_1$ direction, and let $\varphi_t\geq0$ be the norm of $\bm{v}_t$ projected onto the subspace spanned by eigenvectors whose eigenvalue larger than $-\gamma$. Eq. \eqref{v_t_dyn} gives:
	\begin{align*}
	\psi_{t+1} \ge& (1+\gamma \eta)\psi_t -\sigma\sqrt{\psi_t^2 + \varphi_t^2}-2\eta\hat{\epsilon}\\
	\varphi_{t+1} \le &(1+\gamma\eta)\varphi_t + \sigma\sqrt{\psi_t^2 + \varphi_t^2}+2\eta\hat{\epsilon}
	\end{align*}
	where $\sigma = \eta\rho\mathcal{P}(300\hat{c} + 1)$. 
	Noticing that, by choosing $\sqrt{c_{\max}}\le \frac{1}{300\hat{c}+1}\min\{\frac{1}{4}, \frac{1}{4\hat{c}}\}$, and $c\leq c_{max}$, we have for all $t+1<T_2$
	\begin{equation}\label{sigmat_bound}
	4\sigma (t+1) \le 4\sigma T_2 \le 
	4\eta\rho\mathcal{P}(300\hat{c} + 1)\hat{c}\mathcal{T} =4\sqrt{c}(300 \hat{c} + 1)\hat{c}\le 1
	\end{equation}
	Besides, according to the assumption, we have:
	$$\hat{\epsilon}\leq\frac{4-2\sqrt{2}}{4}\frac{c\sqrt{\epsilon^3\rho}}{\chi^3l}\frac{\hat{\delta}}{2\sqrt{d}}(300\hat{c}+1)=\widetilde{O}(\frac{\epsilon^{3+\frac{\theta}{2}}}{d^{\frac{1}{2}(1+\frac{\theta}{2})}})$$
	The is because  $\sqrt{\epsilon^3}\frac{\hat{\delta}}{2\sqrt{d}}=\frac{\sqrt{d}l\epsilon}{\sqrt{\rho}}e^{-\chi}=\frac{\sqrt{d}l\epsilon}{\sqrt{\rho}}\min\{(\frac{c\epsilon^2\delta}{2dl\Delta_f})^{1+\frac{\theta}{4}},e^{-\chi_1}\}=O(\frac{\epsilon^{3+\frac{\theta}{2}}}{d^{\frac{1}{2}(1+\frac{\theta}{2})}})$.
	Also notice that we use the notation $\widetilde{O}$ to hide the $\log(\cdot)$ term coming from $\chi$. By this definition, we have for all $t<T_2$:
	\begin{equation}\label{E15}
		\begin{aligned}
		2\eta\hat{\epsilon}
		&\leq(2-\sqrt{2})\eta\frac{\hat{\delta}}{2\sqrt{d}}\frac{c\sqrt{\epsilon^3\rho}}{\chi^3l}(300\hat{c}+1)\\
		&\overset{(a)}\leq(2-\sqrt{2})\eta\mu r\cdot\frac{\sqrt{c\rho\epsilon}}{\chi}(300\hat{c}+1)\\
		&=(2-\sqrt{2})\mu r\mathcal{P}\rho\eta(300\hat{c}+1)\\
		&=(2-\sqrt{2})\sigma\psi_0
		\end{aligned}			
	\end{equation}
	Where step $(a)$ comes from definition of $\mu$ and $r$. \\
	We will now prove via double induction that for pairs $(t_1,t_2)$, $t_1<T_2$, $t_2<T_2$:
	\begin{equation}\label{E12}
	\varphi_{t_1} \le 4 \sigma_{t_1} \cdot \psi_{t_1}\quad and\quad 2\eta\hat{\epsilon}\leq(2-\sqrt{2})\sigma \psi_{t_2}
	\end{equation}
	By hypothesis of Lemma \ref{lemma_esc}, $\varphi_0 = 0$ and choosing $t=0$ in Eq. \eqref{E15}, we know the base case of induction holds. Assume Eq. \eqref{E12} is true for $(\tau_1,\tau_2)$, where $\tau_1=\tau_2=t\leq T_2$, For $(\tau_1+1,\tau_2+1)=(t+1,t+1)$, $t+1\le T_2$, we have:
	\begin{align*}
	4\sigma(t+1)\psi_{t+1} 
	\ge & 4\sigma (t+1) \left( (1+\gamma \eta)\psi_{t} - \sigma \sqrt{\psi_{t}^2 + \varphi_{t}^2}-2\eta\hat{\epsilon}\right) \\
	\varphi_{t+1} \le &4 \sigma  t(1+\gamma\eta) \psi_{t} + \sigma \sqrt{\psi_{t}^2 + \varphi_{t}^2}+2\eta\hat{\epsilon}
	\end{align*} 
	To derive the first equation, we multiply $4\sigma(t+1)$ on both sides and to get the second equation, we use induction when $\tau_1=t$.\\
	Based on induction when $\tau_1=t$ and Eq. \eqref{sigmat_bound}, we know that $\varphi_{t}\leq4\sigma t\cdot\psi_{t}\leq\psi_{t}$. To finish the induction, we only need to show:
	$$	4 \sigma  t(1+\gamma\eta) \psi_{t} + \sigma \sqrt{\psi_{t}^2 + \varphi_{t}^2}+2\eta\hat{\epsilon}\leq4\sigma (t+1) \left( (1+\gamma \eta)\psi_{t} - \sigma \sqrt{\psi_{t}^2 + \varphi_{t}^2}-2\eta\hat{\epsilon}\right)$$
	Which means we only need to show
	\begin{equation*}
		\left(1+4\sigma (t+1)\right)[\sigma\sqrt{\psi_{t}^2 + \varphi_{t}^2}+2\eta\hat{\epsilon}]
	\le 4 (1+\gamma \eta)\sigma\psi_{t}
	\end{equation*}
	Recall that $\varphi_t \le 4  \mu t \cdot \psi_t \le \psi_t$, combine with Eq. \eqref{sigmat_bound} and use the induction assumption when $\tau_2=t$, we have
	\begin{align*}
	\left(1+4\sigma (t+1)\right)[\sigma\sqrt{\psi_{t}^2 + \varphi_{t}^2}+2\eta\hat{\epsilon}]
	&\leq\left(1+4\sigma (t+1)\right)\big[\sigma\sqrt{2\psi_{t}^2}+2\eta\hat{\epsilon}\big]\\
	&\leq 2\sqrt{2}\sigma\psi_{t}+(4-2\sqrt{2})\sigma\psi_{t}\\
	&=4\sigma\psi_{t}<4(1+\gamma\eta)\sigma\psi_{t}
	\end{align*}
	which finishes the proof for $\tau_1=t+1$.\\
	Recall that $\varphi_t \le 4  \mu t \cdot \psi_t \le \psi_t$, again use the induction assumption when $\tau_2=t$, we have
	\begin{equation}\label{power_increase}
	\psi_{t+1} \ge (1+\gamma \eta)\psi_t - \sqrt{2}\sigma\psi_t -(2-\sqrt{2})\sigma\psi_t=(1+\gamma\eta)\psi_t-2\sigma\psi_t
	\ge (1+\frac{\gamma \eta}{2})\psi_t 
	\end{equation}
	where the last step follows from $\sigma = \eta\rho\mathcal{P}(300\hat{c}+ 1) \le  \sqrt{c_{\max}}(300\hat{c} + 1) \gamma \eta \cdot\chi^{-1} < \frac{\gamma \eta}{4}$.\\
	This mean $\psi_{t+1}\geq\psi_t$. Combing with Eq. \eqref{E15}, we finish the proof for $\tau_2=t+1$. Thus, we finish the whole double induction.
	
	Finally, combining Eq. \eqref{vt_bound} and \eqref{power_increase},  we have for all $t<T_2$:
	\begin{align*}
	200(\mathcal{P}\cdot\hat{c})
	\ge &\norm{\bm{v}_t} \ge \psi_t \ge (1+\frac{\gamma \eta}{2})^t \psi_0
	\ge (1+\frac{\gamma \eta}{2})^t \frac{\delta}{2\sqrt{d}}r
	= (1+\frac{\gamma \eta}{2})^t \frac{\hat{\delta}}{2\sqrt{d}}\frac{\mathcal{P}}{\kappa\chi}
	\end{align*}
	Noticing that $\frac{\eta\gamma}{2}=\frac{c\gamma}{2l}=\frac{c}{2k}<1$, and we have for $x\in(0,1)$, $\log(1+x)>\frac{x}{2}$. Choosing $t=\frac{T_2}{2}<T_2$ in above equation, this implies:
	\begin{equation*}
	\begin{aligned}
	T_2 &< 2\frac{\log(400\frac{k\sqrt{d}}{\hat{\delta}}\cdot\hat{c}\chi)}{\log(1+\frac{\eta\gamma}{2})} <  8\frac{\log(400\frac{kd}{\hat{\delta}}\cdot\hat{c}\chi)}{\eta\gamma}=8\frac{\log(400\hat{c})+\log(\chi)+\log(\frac{\kappa d}{\hat{\delta}})}{\eta\gamma}\\
	&\overset{(a)}=8\frac{\log(400\hat{c})+\log(\chi)+\chi}{\eta\gamma}
	\le 8\frac{\log(400\hat{c})\chi+\chi+\chi}{\eta\gamma}
	=8(2+\log(400\hat{c}))\frac{\chi}{\eta\gamma}
	=8(2+\log(400\hat{c}))\mathcal{T}
	\end{aligned}
	\end{equation*}
	Notice that $\log(\frac{d\kappa}{\hat{\delta}})=\log(\frac{d\kappa\sqrt{\rho\epsilon}}{dl}e^{\chi})=\log(e^{\chi})=\chi$. By choosing constant $\hat{c}$ to be large enough to satisfy $8(2 + \log (400 \hat{c})) \le \hat{c}$, we will have
	$T_2 < \hat{c}\mathcal{T} $, which finishes the proof.
\end{proof}

%% file: lem4proof.tex
\section{Proof of Lemma \ref{lemma_esc}}\label{apdlem4}

\begin{proof}
	W.L.O.G, let $\widetilde{\bm{x}} = 0$ be the origin. Let $(c^{(2)}_{\max}, \hat{c})$ be the absolute constant so that Lemma \ref{lemma_wt} holds, also let $c^{(1)}_{\max}$ be the absolute constant to make Lemma \ref{lemma_ut} holds based on our current choice of $\hat{c}$.
	We choose $c_{\max} \le \min\{c^{(1)}_{\max}, c^{(2)}_{\max}\}$ so that our learning rate $\eta \le c_{\max}/\ell$ is small enough which makes both Lemma \ref{lemma_ut} and Lemma \ref{lemma_wt} hold. Let $T^{*}:=\hat{c}\mathcal{T}$ and define:
	\begin{equation*}
	T' = \inf_t\left\{t| \widetilde{f}_{\bm{u}_0}(\bm{u}_t)-f(\bm{u}_0)\le-4.5f_{thres} \right\}
	\end{equation*}
	Let's consider following two cases:
	
	\paragraph{Case $T' \le T^{*}$:} In this case, by Lemma \ref{lemma_ut}, we know $\norm{\bm{u}_{T'-1}} \le O(\mathcal{P})$. Using $l$-smooth, we have
	\begin{equation*}
		\begin{aligned}
		\Vert\bm{u}_{T'}\Vert
		&\overset{(a)}\leq\Vert\bm{u}_{T'-1}\Vert+\eta\Vert\hat{\nabla}f(\bm{u}_{T'-1})\Vert
		\leq\Vert\bm{u}_{T'-1}\Vert+\eta\Vert\nabla f(\bm{u}_{T'-1})\Vert+\eta\Vert\hat{\nabla}f(\bm{u}_{T'-1})-\nabla f(\bm{u}_{T'-1})\Vert\\
		&\overset{(b)}\leq\Vert\bm{u}_{T'-1}\Vert+\eta\Vert\nabla f(\widetilde{\bm{x}})\Vert+\eta l\Vert\bm{u}_{T'-1}\Vert+\eta\hat{\epsilon}
		\overset{(c)}\leq 2(\Vert\bm{u}_{T'-1}\Vert+\eta g_{thres})
		\overset{(d)}\leq O(\mathcal{P})
		\end{aligned}
	\end{equation*}
	where $(a)$ comes from the gradient descent step in Algorithm \ref{PGD-MF}, $(b)$ uses the $l$-smooth property, $(c)$ follows the definition of $\widetilde{\bm{x}}$ and $\hat{\epsilon}$ and $\eta g_{thres}\leq\frac{\sqrt{c}}{\chi^2}\cdot\frac{\epsilon}{l}=\frac{\sqrt{\epsilon\rho}}{l\chi}(\frac{\sqrt{c}}{\chi}\sqrt{\frac{\epsilon}{\rho}})=\frac{1}{\chi\kappa}\mathcal{P}\leq\mathcal{P}$ gives the inequality $(e)$.\\
	Using this, we can the function decrease greatly from $\bm{u}_0$ to $\bm{u}_{T'}$
	\begin{equation*}
		\begin{aligned}
		f(\bm{u}_{T'}) - f(\bm{u}_0) 
		&\overset{(a)}\le \nabla f(\bm{u}_0)^T(\bm{u}_{T'}-\bm{u}_0) + \frac{1}{2}(\bm{u}_{T'}-\bm{u}_0)\trans \nabla^2 f(\bm{u}_0) (\bm{u}_{T'}-\bm{u}_0)
		+ \frac{\rho}{6} \norm{\bm{u}_{T'}-\bm{u}_0}^3 \\
		&\overset{(b)}=\widetilde{f}_{\bm{u}_0}(\bm{u}_{T'})-f(\bm{u}_0)+\frac{1}{2}(\bm{u}_{T'}-\bm{u}_0)\trans [\nabla^2f(\bm{u}_0)-\nabla^2f(\widetilde{\bm{x}})](\bm{u}_{T'}-\bm{u}_0)+ \frac{\rho}{6} \norm{\bm{u}_{T'}-\bm{u}_0}^3\\
		&\overset{(c)}\le \widetilde{f}_{\bm{u}_0}(\bm{u}_{T'})-f(\bm{u}_0)+\frac{\rho}{2}\Vert\bm{u}_0-\widetilde{\bm{x}}\Vert\Vert\bm{u}_{T'}-\bm{u}_0\Vert^2 + \frac{\rho}{6} \norm{\bm{u}_{T'}-\bm{u}_0}^3 \\
		&\overset{(d)}\le-4.5f_{thres} + O(\rho\mathcal{P}^3) \overset{(e)}= -4.5f_{thres} + O(\sqrt{c}\cdot f_{thres}) \overset{(f)}\le -4f_{thres}
		\end{aligned}
	\end{equation*}
	where $(a)$ and $(c)$ directly use $\rho$-Hessian Lipschitz, $(b)$ comes from the definition of $\widetilde{f}_{\bm{u}_0}(\bm{u}_{T'})$, $(d)$ follows the Lemma \ref{lemma_ut} and $\rho\mathcal{P}^3=\frac{(c\epsilon)^{1.5}}{\chi^3\sqrt{\rho}}=\sqrt{c}\frac{c}{\chi^3}\sqrt{\frac{\epsilon^3}{\rho}}=\sqrt{c}f_{thres}$ give the inequality $(e)$. Finally, by choosing $c$ small enough, the inequality $(f)$ holds.\\
	Now, we are going to bound the increase of function from step $T'$ to $T$. Because when $\Vert\hat{\nabla}f(\bm{x}_t)\Vert>g_{thres}$, the function value will decrease by Lemma \ref{lemma_leq}. Thus, we only consider the condition that $\Vert\hat{\nabla}f(\bm{x}_t)\Vert\leq g_{thres}$. According to Eq. \eqref{Decrease} step (c) in Lemma \ref{lemma_geq}, by choosing $\hat{\epsilon}\leq cg_{thres}=O(\epsilon)$, we have
	\begin{equation}\label{Decrease2}
	\begin{aligned}
	f(\bm{u}_{t+1}) - f(\bm{u}_t) &\le \eta\Vert\nabla f(\bm{u}_t)-\hat{\nabla}f(\bm{u}_t)\Vert\Vert\hat{\nabla}f(\bm{u}_t)\Vert+\frac{l\eta^2}{2}\Vert\hat{\nabla} f(\bm{u}_t)\Vert^2-\eta\Vert\hat{\nabla} f(\bm{u}_t)\Vert^2\\
	&\overset{(a)}\le \eta\hat{\epsilon}\Vert\hat{\nabla} f(\bm{u}_t)\Vert+\frac{c\eta}{2}\Vert\hat{\nabla} f(\bm{u}_t)\Vert^2\\
	&\le \eta cg_{thres}^2+\frac{c\eta}{2}g_{thres}^2=\frac{3}{2}c\eta g_{thres}^2 
	\end{aligned}
	\end{equation}
	where we omit the non-positive term in step $(a)$.\\ 
	Choosing $c_{\max} \le \min \{1, \frac{1}{\hat{c}}\}$. We know $T=\frac{\mathcal{T}}{c}\geq\frac{\mathcal{T}}{c_{max}}\geq\hat{c}\mathcal{T}=T^*\geq T'>0$. Thus, the number of steps between $T$ and $T'$ are at most $\frac{\mathcal{T}}{c}$. Therefore, during these steps, the function value can at most increase:
	\begin{equation}
	\begin{aligned}
	f(\bm{u}_T)-f(\bm{u}_{T'})\leq\big(f(\bm{u}_t) - f(\bm{u}_{t+1})\big)\frac{\mathcal{T}}{c}
	&\leq \frac{3}{2}c\eta g_{thres}^2\frac{\chi}{c\eta\gamma}
	=\frac{3}{2}\frac{c}{\chi^4}\epsilon^2\frac{\chi}{\sqrt{\rho\epsilon}}\\
	&\leq\frac{3c}{2\chi^3}\cdot\sqrt{\frac{\epsilon^3}{\rho}}=1.5f_{thres}\\
	\end{aligned}
	\end{equation}
	Thus, we have:
	\begin{equation*}
	f(\bm{u}_T) - f(\bm{u}_0) = [f(\bm{u}_T) - f(\bm{u}_{T'})] + [f(\bm{u}_{T'}) - f(\bm{u}_0)] \le 1.5f_{thres} - 4f_{thres}= -2.5f_{thres}
	\end{equation*}
	
	\paragraph{Case $T' > T^*$:} In this case, by Lemma \ref{lemma_ut}, we know $\norm{\bm{u}_t}\le O(\mathcal{P})$ for all $t\le T^*$. Define 
	\begin{equation*}
	T'' = \inf_t\left\{t| \tilde{f}_{\bm{w}_0}(\bm{w}_t) - f(\bm{w}_0)  \le -4.5f_{thres} \right\}
	\end{equation*}
	Noticing that $\Vert\bm{w}_0-\widetilde{\bm{x}}\Vert\leq\Vert\bm{u}_0+\mu r\bm{e}_1\Vert\leq2r$. By Lemma \ref{lemma_ut}, we have for $t<T_2$, $\Vert \bm{w}_t-\widetilde{\bm{x}}\Vert\leq100(\hat{c}\cdot\mathcal{P})$, which is exactly the condition in Lemma \ref{lemma_wt}. Thus, by Lemma \ref{lemma_wt}, we immediately have $T'' \le T^*$. Applying same argument as in first case (replacing notation $\bm{u}$ with $\bm{w}$), we have for $T=t_{thres}=\frac{\mathcal{T}}{c}$ that $f(\bm{w}_T) - f(\bm{w}_0) \le -2.5f_{thres}$.
\end{proof}

%% file: prooflem3.tex
\section{Proof of Lemma \ref{lemma_leq}}\label{apdlem3}
Using $l$-smooth, by adding a perturbation, we know the function value at most increase by
\begin{equation}\label{Decrease2}
	\begin{aligned}
	f(\bm{x}_0) - f(\widetilde{\bm{x}}) 
	\le \nabla f(\widetilde{\bm{x}})^T\bm{\xi}+\frac{l}{2}\Vert\bm{\xi}\Vert^2
	\le g_{thres}r+\frac{lr^2}{2}
	\le \frac{3}{2}f_{thres}
	\end{aligned}
\end{equation}
Where $(a)$ holds due to $\hat{\epsilon}$ close gradient estimation, $\eta=\frac{c}{l}$, and omitting the last non-positive term

By applying Lemma \ref{lemma_esc}, we know for any $\bm{x}_0\in \cXs$, it is guaranteed that $(\bm{x}_{0} \pm \mu r \bm{e}_1) \not \in \cXs $, where $\mu \in [\frac{\hat{\delta}}{2\sqrt{d}}, 1]$. Let $I_{\cXs}(\cdot)$ be the indicator function of being inside set $\cXs$; and vector $\bm{x} = (x^{(1)}, \bm{x}^{(-1)})$, where $x^{(1)}$ is the component along $\bm{e}_1$ direction, and $\bm{x}^{(-1)}$ is the remaining $d-1$ dimensional vector. Define $\mathbb{B}^{(d)}(r)$ be $d$-dimensional ball with radius $r$. We obtain an upper bound on the volume of $\cXs$ as follows. 
\begin{eqnarray}
\text{Vol}(\cXs) \nonumber
&= & \int_{\mathbb{B}^{(d)}_{\tilde{\bm{x}}}(r)}  \mathrm{d}\bm{x} \cdot I_{\cXs}(\bm{x})\nonumber\\
&= & \int_{\mathbb{B}^{(d-1)}_{\tilde{\bm{x}}}(r)} \mathrm{d} \bm{x}^{(-1)} \int_{y_l}^{y_u} \mathrm{d} x^{(1)}  \cdot  I_{\cXs}(\bm{x})\nonumber\\
&\le & \int_{\mathbb{B}^{(d-1)}_{\tilde{\bm{x}}}(r)} \mathrm{d} \bm{x}^{(-1)} \cdot\left(2\cdot \frac{\hat{\delta}}{2\sqrt{d}}r \right) \nonumber\\
&=& \text{Vol}(\mathbb{B}_0^{(d-1)}(r))\times \frac{\hat{\delta} r}{\sqrt{d}},
\end{eqnarray}
where $y_l=\tilde{x}^{(1)} - \sqrt{r^2 - \norm{\tilde{\bm{x}}^{(-1)} - \bm{x}^{(-1)}}^2}$, and $y_u=\tilde{x}^{(1)} + \sqrt{r^2 - \norm{\tilde{\bm{x}}^{(-1)} - \bm{x}^{(-1)}}^2}$.

We next obtain an upper bound on $\frac{\text{Vol}(\cXs)}{\text{Vol}(\mathbb{B}^{(d)}_{\tilde{\bm{x}}}(r))}$ as follows.

\begin{eqnarray}
\frac{\text{Vol}(\cXs)}{\text{Vol}(\mathbb{B}^{(d)}_{\tilde{\bm{x}}}(r))}&
\le& \frac{\frac{\hat{\delta} r}{\sqrt{d}} \times \text{Vol}(\mathbb{B}^{(d-1)}_0(r))}{\text{Vol} (\mathbb{B}^{(d)}_0(r))}\nonumber\\
&=& \frac{\hat{\delta}}{\sqrt{\pi d}}\frac{\Gamma(\frac{d}{2}+1)}{\Gamma(\frac{d}{2}+\frac{1}{2})} \nonumber\\
&\le& \frac{\hat{\delta}}{\sqrt{\pi d}} \cdot \sqrt{\frac{d}{2}+\frac{1}{2}}\nonumber \\
&\le& \hat{\delta}
\end{eqnarray}
The second last inequality is by the Gautschi's inequality \citep{elezovic2000best}, which states that $\frac{\Gamma(x+1)}{\Gamma(x+1/2)}<\sqrt{x+\frac{1}{2}}$ as long as $x\ge 0$.
Due to $\bm{\xi}$ chosen from uniform distribution ball with radius $r$ , therefore, with at least probability $1-\hat{\delta}$, $\bm{x}_0 \not \in \cXs$. Thus, by Lemma \ref{lemma_esc} 
\begin{align*}
f(\bm{x}_T) - f(\tilde{\bm{x}}) =& f(\bm{x}_T)  - f(\bm{x}_0) +  f(\bm{x}_0) - f(\widetilde{\bm{x}}) \\
\le & -2.5f_{thres} + 1.5f_{thres} \le -f_{thres}
\end{align*}
which completes the proof of  Lemma \ref{lemma_leq}.